\documentclass[reqno,oneside,12pt]{amsart}

\usepackage[utf8]{inputenc}
\usepackage[english]{babel}
\usepackage{amsmath, amsfonts, amsthm, enumerate, array, latexsym, stmaryrd, amssymb}
\usepackage[all, poly]{xy}
\usepackage{color}
\definecolor{darkgreen}{rgb}{0,0.45,0}
\usepackage[pagebackref,colorlinks,citecolor=darkgreen,linkcolor=darkgreen]{hyperref}
\usepackage{graphicx}
\usepackage{stmaryrd}
\usepackage[normalem]{ulem}

\newtheorem{theo}{Theorem}[section]
\newtheorem{lem}[theo]{Lemma}
\newtheorem{theorem}[theo]{Theorem}
\newtheorem{lemma}[theo]{Lemma}
\newtheorem{prop}[theo]{Proposition}
\newtheorem{proposition}[theo]{Proposition}
\newtheorem{cor}[theo]{Corollary}

\theoremstyle{definition}
\newtheorem{defi}[theo]{Definition}
\newtheorem*{ex*}{Example}
\newtheorem{ex}[theo]{Example}

\newcommand{\pr}{\text{pr}}

\renewcommand{\l}{\left}
\renewcommand{\r}{\right}
\renewcommand{\v}{\vert}

\newcommand{\la}{\langle}
\newcommand{\ra}{\rangle}
\newcommand{\bmat}{\begin{pmatrix}}
\newcommand{\emat}{\end{pmatrix}}
\newcommand{\bdet}{\begin{vmatrix}}
\newcommand{\edet}{\end{vmatrix}}
\newcommand{\barray}{\l\{\begin{array}}
\newcommand{\earray}{\end{array}\r.}

\newcommand{\st}{\,\v\,}
\renewcommand{\t}[1]{\text{#1}}

\def\hot{\widehat{\otimes}}
\renewcommand{\o}{\otimes}

\newcommand{\Gal}{{\sf Gal}}
\newcommand{\Perm}{{\sf Perm}}

\newcommand{\selabel}[1]{\label{se:#1}}

\newcommand{\prref}[1]{Proposition~\ref{pr:#1}}

\def\ul{\underline}
\newcommand{\Hom}{{\sf Hom}}

\newcommand{\End}{{\sf End}}

\newcommand{\Mod}{{\sf Mod}}

\newcommand{\can}{{\ul{\sf can}}}

\newcommand{\Vect}{{\sf Vect}}
\newcommand{\ProVect}{{\sf ProVect}}
\newcommand{\Alg}{{\sf Alg}}

\newcommand{\res}{{\textrm{res}}}

\def\ot{\otimes}

\def\id{\textrm{{\small 1}\normalsize\!\!1}}

\newcommand{\ol}[1]{\overline{#1}}

\def\NN{{\mathbb N}}

\def\QQ{{\mathbb Q}}

\newcommand{\plim}[1]{\underset{#1}{\underleftarrow{\lim} \;}}
\newcommand{\ilim}[1]{\underset{#1}{\underrightarrow{\lim} \;}}

\newcommand{\exclude}[1]{}

\begin{document}

\title{Correspondence theorems for infinite Hopf-Galois extensions}
\author{Hoan-Phung Bui}
\email{\href{mailto:hbui@ulb.ac.be}{\texttt{hoanphung.bui@gmail.com}}}
\author{Joost Vercruysse}
\address{Département de Mathématique, Universit\'e Libre de Bruxelles (ULB), B-1050 Bruxelles, Belgium}
\email{\href{mailto:joost.vercruysse@ulb.be}{\texttt{joost.vercruysse@ulb.be}}}
\author{Gabor Wiese}
\address{Département Mathématiques, Université du Luxembourg, L-4364 Esch-sur-Alzette, Luxembourg}
\email{\href{mailto:gabor.wiese@uni.lu}{\texttt{gabor.wiese@uni.lu}}}

\subjclass[2010]{12F10, 20E18, 16S40, 16T05}

\date{\footnotesize \today}

\begin{abstract}
This paper extends Hopf-Galois theory to infinite field extensions and provides a natural definition of subextensions.
For separable (possibly infinite) Hopf-Galois extensions, it provides a Galois correspondence.
This correspondence also is a refinement of what was known in the case of finite separable Hopf-Galois extensions.
\end{abstract}

\maketitle

\section{Introduction}

Hopf-Galois theory arose from the quest to generalise the classical Galois theory of field extensions to commutative rings \cite{CHR}, \cite{ChaseSweedler69} and even non-commutative ring extensions \cite{KreimerTakeuchi81}, but also leads to a new approach to study classical (separable) field extensions \cite{GreitherPareigis87}, see \cite{HGbook} for a recent overview on the topic.
From the very beginning, there has been a lot of interest in extending the classical Galois correspondence between intermediate field extensions and subgroups of the Galois group to the setting of Hopf-Galois extensions. It is well-known (see e.g.\ \cite{Schauenburg98}) that, in full generality, such a correspondence does not hold: in general, the correspondence map from Hopf subalgebras to intermediate extensions is injective but not surjective. 
Whereas often in the literature, one has been looking for conditions identifying classes of Hopf-Galois extensions for which the correspondence between intermediate fields and Hopf subalgebras becomes bijective, we take a different approach here, and characterize for an arbitrary separable Hopf-Galois extension, those intermediate extensions for which the correspondence holds.

More precisely, the aim of the present paper is twofold.
We extend Hopf-Galois theory to infinite field extensions and provide a natural definition of {\em $H$-subextensions}, for which we prove a Galois correspondence theorem for separable Hopf-Galois extensions.
This leads to new results already in the case of finite separable field extensions. Whilst our proof establishes the finite case first, in this introduction we present the general infinite statements immediately.

For our formulation of infinite Hopf-Galois theory we need the notion of proartinian Hopf algebras and proartinian coalgebras. These have been considered earlier (for example, by Fontaine \cite{Fontaine}) and are basically limits of projective systems of finite dimensional Hopf algebras or coalgebras, respectively. As such they carry a natural topology and the discrete ones are exactly those of finite dimension. It is further natural to introduce proartinian $H$-modules and proartinian $H$-module algebras. The reader is referred to Section~\ref{sec:proartinian} for a presentation of this theory.

Classical (infinite) Galois theory gives an isomorphism $L\llbracket G \rrbracket \cong L \hot_K K \llbracket G \rrbracket \cong \End_K(L)$ if $L/K$ is a Galois extension (i.e.\ separable and normal) with (profinite) Galois group~$G$.
Hopf-Galois theory generalises this as follows:
The full endomorphism ring of a field extension $L/K$ is already `encoded' in a (proartinian) Hopf algebra~$H$ that is defined over the base field~$K$.

\begin{defi}\label{definfiniteHG}
Let $L/K$ be a field extension, $H$ a proartinian Hopf $K$-algebra and $L$ a discrete left $H$-module algebra.
This action leads to the definition of a {\em canonical map} (see \eqref{defcanproartinian})
\[ \can: L \hot_K H \to \End_K(L). \]
We say that the extension $L/K$ is {\em $H$-Galois} if the canonical map $\can$ is a homeomorphism of topological $L$-vector spaces.
\end{defi}

We include a short digression on our interpretation of the definition, as it paves the way for our definition of $H$-subextensions, leading to our correspondence theorem.
For a field extension $L/K$ and a proartinian $L$-vector space~$V$, we say that a proartinian $K$-vector space~$W$ is a {\em $K$-rational structure of~$V$} if $L \hot_K W$ is a proartinian $L$-vector space isomorphic to~$V$. In the case of finite dimensional vector spaces, $W\subset V$ is a $K$-rational structure of~$V$ if any $K$-basis of~$W$ is an $L$-basis of~$V$.

From this perspective, in an $H$-Galois extension $L/K$, the proartinian Hopf algebra provides a $K$-rational structure of $\End_K(L)$, just as $K\llbracket G \rrbracket$ provides a $K$-rational structure of $L\llbracket G \rrbracket \cong \End_K(L)$ in the case of a classical Galois extension.
This idea can be naturally carried over to any intermediate field $K\subset L_0\subset L$. For that purpose, we define the {\em annihilator of $L_0$ in~$H$} by
\begin{equation}\label{defi:JL0}
J(L_0) = \{h \in H \;|\; \forall\, x \in L_0: h \cdot x = 0 \}.
\end{equation}
It is a closed left ideal two-sided proartinian coideal of $H$ (see Lemma~\ref{le:propertiesJ(L0)}).
Via the canonical map, we can view $H/J(L_0)$ inside $\Hom_K(L_0,L)$. The principal idea is that $H/J(L_0)$ will play the same role for subextensions as $H$ does for the entire Hopf-Galois extension, namely that of a $K$-rational structure.

\begin{defi}\label{defi:Hsubext}
Let $L/K$ be $H$-Galois and $L_0$ be an intermediate field.
We say that $L_0/K$ is an {\em $H$-subextension} if $H/J(L_0)$ is a $K$-rational structure of $\Hom_K(L_0,L)$ in the sense that the canonical map
\[\can_0: L \hot H/J(L_0) \to \Hom_K(L_0,L)\]
is a homeomorphism. 
\end{defi}

In fact, as we will show (see Lemma~\ref{le:propertiesJ(L0)}), it is enough to assume that $\can_0$ is injective, as it is already a quotient map of topological vector spaces.
The classical notion of normal field extensions leads in our setting to the following natural definition.

\begin{defi}\label{defi:stable}
Let $L/K$ be $H$-Galois.
We say that an intermediate field $L_0$ is {\em $H$-stable} if $H \cdot L_0 \subseteq L_0$.
If, furthermore, $L_0$ is an $H$-subextension, then we call it {\em $H$-normal}.
\end{defi}

In classical Galois theory, subextensions are obtained and characterised as fixed fields.
This point of view carries on to Hopf-Galois theory via the following definition.

\begin{defi}\label{defi:fix}
Let $L$ be a left $H$-module $K$-algebra.
Given any subset $F \subset H$, we define the {\em fixed space by~$F$} or the space of {\em $F$-invariants} as
\[ L^F := \big\{x \in L \st h\cdot x = \epsilon(h)x \quad \forall h \in F\big\}.\]
\end{defi}
Remark that $L^F$ is indeed a $K$-subspace of~$L$. If $F$ is a set of grouplike elements, then $\epsilon(g)=1$, and $L^F$ is a fixed field in the classical sense. On the other hand, if $F$ is a proartinian coideal in $H$, then $\epsilon(F)=0$, and $L^F$ consists of elements $x\in L$ satisfying $h\cdot x=0$ for all $h\in F$.

Since, as in classical Galois theory, our correspondence theorem for subfields passes via morphisms, we also have to assume the separability for the field extension, which ensures the existence of `enough' field morphisms.
The crucial point for us is that for a separable field extension $L/K$ with normal closure (or algebraic closure) $\tilde L$ the coalgebra $\Hom_K(L,\tilde L)$ has a (topological) $\tilde L$-basis of grouplike elements, namely exactly given by the field morphisms $L \to \tilde L$.
In the case of an $H$-Galois extension, via the canonical map, this then implies that $\tilde L \hot_K H$ is a (completed) group algebra.
The same arguments actually also give the converse: any $H$-extension such that $\tilde L \hot_K H$ is a (completed) group algebra is separable.
In our understanding, in the finite case, this use of separability is exactly the starting point of Greither-Pareigis theory (\cite{GreitherPareigis87}, see also Theorem~\ref{th:GP} below), many ideas of which we crucially apply and extend to the infinite case.

We now have all ingredients to state our main theorem.

\begin{theo}\label{theo:inf-main}
Let $L/K$ be a separable $H$-Galois extension for a proartinian Hopf algebra~$H$. Then the following maps are inclusion reversing bijections:
\begin{equation}\label{eq:corr:inf}
\xymatrix@R0.5cm{
\big\{L/L_0/K \st L_0 \t{ $H$-subextension}\big\} \ar@<.5ex>[r]^-\Phi & \big\{I \subset H \st I \t{ closed left ideal, two-sided coideal}\big\} \ar@<.5ex>[l]^-{\Psi} \\
L_0 \ar@{|->}[r] & J(L_0)\\
L^I  & I \ar@{|->}[l]\\
}
\end{equation}
Moreover, the above correspondence restricts to a bijection between the following subsets
\[
\xymatrix@R0.5cm{
\big\{L/L_0/K \st L_0 \t{ $H$-normal}\big\} \ar@<.5ex>[r]^-\Phi
& \big\{I \subset H \st I \t{ closed Hopf ideal}\big\}
\ar@<.5ex>[l]^-{\Psi} \\
}\]
Furthermore, if `closed' is replaced by `open', both correspondences work with the restriction that $L_0/K$ be finite.
\end{theo}

In the case of classical Galois extensions, one can freely pass between considering normal subgroups $H \lhd G$ of the Galois group and the corresponding quotients $G/H$. This point of view also exists for cocommutative Hopf algebras (see Theorem~\ref{theo:subalgebra-ideal} and Corollary~\ref{cor:subalgebra-ideal}). In particular, in the case of a finite extension $L/K$, the sets on the right in the above correspondences can be replaced by the set of Hopf subalgebras of $H$, and the set of normal Hopf subalgebras, respectively.

Whereas most of the lemmas needed to prove our Hopf-Galois correspondence are elementary, as already alluded to above, Greither-Pareigis theory, extended to infinite separable Hopf-Galois extensions, plays a key role in proving the correspondence of Theorem~\ref{theo:inf-main}.
It additionally also provides an alternative formulation of it in purely group-theoretical terms.
Greither-Pareigis theory relates infinite separable Hopf-Galois extensions $L/K$ with strictly transitive actions of profinite groups on the set of $K$-embeddings into a normal closure $\tilde L$ of~$K$, see Proposition~\ref{prop:GPinf}. Particularly, the separability leads to an isomorphism $\tilde L \hot H \cong \tilde L \llbracket N \rrbracket$ with a profinite group~$N$, equipped with a $G$-action.

\begin{cor}\label{co:correspGPinf}
There is an explicit bijective correspondence:
\[
\xymatrix@R0.5cm{
\big\{L/L_0/K \st L_0 \t{ $H$-subextension }\big\} \ar@<.5ex>[r]^-{\Phi'} & \big\{V \subset N \st V \t{ is a closed $G$-equivariant subgroup}\big\} \ar@<.5ex>[l]^-{\Psi'} 
}\]
restricting to another bijective correspondence
\[
\xymatrix@R0.5cm{
\big\{L/L_0/K \st L_0 \t{ $H$-normal }\big\} \ar@<.5ex>[r]^-{\Phi'}
& \big\{V \subset N \st V \t{ is a closed $G$-equivariant normal subgroup}\big\}
\ar@<.5ex>[l]^-{\Psi'} \\
}\]
The $G$-equivariance of the subgroup can also be rephrased as being normalized by $\lambda(G)$ (the group of left translations via $G$) in $\Perm(G/G')$.
\end{cor}

Moreover, in the finite case, we obtain in section \ref{HGtheory} several other analogues of statements from classical Galois theory in the setting of Hopf-Galois extensions. For example, we show that the intersection and compositum of Hopf-Galois subextensions is again a Hopf-Galois subextension (see Proposition~\ref{prop:intercomp}).

Finally, we show that the Hopf-Galois condition in the infinite setting, fits in the framework of Hopf-Galois extensions in terms of coactions (of usual Hopf algebras) rather than actions, by taking a suitable restricted dual of our proartinian Hopf algebra, Proposition~\ref{prop:coaction}.

\subsection*{Notational conventions}

Throughout, $K$ will denote a field. Unadorned tensor products are tensor products over~$K$. The action of the Hopf algebra is denoted by $\cdot$ everywhere, multiplication inside $H$ and $L$ by concatenation (no symbol). In order to be compatible in the case of a group algebra, group actions are also written with $\cdot$. If the group is a group of automorphisms, then the product of two elements is also denoted $\circ$. The group algebra of a group $G$ over the field $K$ will be denoted $K[G]$.

\section{Finite Hopf-Galois Theory}

\subsection{Preliminaries on Hopf algebras and some first observations}\label{preliminariesHopf}

Recall that a $K$-coalgebra $C$ is a vector space endowed with a comultiplication map $\Delta:C\to C\ot C$ and a counit map $\epsilon:C\to K$ satisfying the usual coassociativity and counitality conditions. We will use the Sweedler notation for comultiplication: $\Delta(c)=c_{(1)}\ot c_{(2)}$, for all $c\in C$, so that the coassociativity and counitality conditions can be expressed as 
\begin{eqnarray*}
&c_{(1)(2)}\ot c_{(1)(2)}\ot c_{(2)} = c_{(1)}\ot c_{(2)(1)}\ot c_{(2)(2)} = c_{(1)}\ot c_{(2)}\ot c_{(3)};\\
&c_{(1)}\epsilon(c_{(2)})=c=\epsilon(c_{(1)}) c_{(2)}.
\end{eqnarray*}
A coalgebra is called cocommutative if $c_{(1)}\ot c_{(2)}=c_{(2)}\ot c_{(1)}$ for all $c\in C$. A grouplike element in a coalgebra $C$ is an element $x\in C$ such that $\Delta(x)=x\ot x$ and $\epsilon(x)=1$.
A $K$-bialgebra $H$ is a $K$-algebra that has a coalgebra structure such that $\Delta$ and $\epsilon$ are $K$-algebra homomorphisms. A Hopf $K$-algebra is a $K$-bialgebra for which there exists a (unique) antipode map $S:H\to H$ satisfying $S(h_{(1)})h_{(2)}=\epsilon(h)1=h_{(1)}S(h_{(2)})$. A group algebra $K[G]$ for a finite group~$G$ is a Hopf algebra by defining the coalgebra structure such that all elements of the group $G$ are grouplike.

A two-sided coideal in a $K$-coalgebra $H$ is a $K$-subspace $I\subset H$ such that $\Delta(I)\subset H\ot I+I\ot H$ and $\epsilon(I)=0$. A Hopf ideal in a Hopf $K$-algebra $H$ is a two-sided ideal and two-sided coideal that is stable under the antipode. It is well-known that, for example, when $H$ is cocommutative, any two-sided ideal and two-sided coideal is stable under the antipode, i.e. is a Hopf ideal (see \cite{Nichols}). A Hopf subalgebra $H_0$ of $H$ is called normal if $h_{(1)}xS(h_{(2)})$ and $S(h_{(1)})xh_{(2)}$ belong to $H_0$ for all $x\in H_0 $ and all $h\in H$.

\begin{defi}\label{defi:coinvariant-augmentation}
Let $I \subset H$ be a left ideal two-sided coideal and denote by $\pi : H \to H/I$ the canonical surjection. We define the set of left $H/I$-coinvariants of $H$:
\[{}^{\text{co}\,H/I}H := \big\{h \in H \;|\; \pi(h_{(1)}) \ot h_{(2)} = \pi(1_H) \ot h\big\}.\]
For a Hopf subalgebra $A \subset H$, we define the augmentation ideal
\[A^+ := A \cap \ker \epsilon = \big\{a \in A \;|\; \epsilon(a) = 0\big\}.\]
\end{defi}

\begin{theo}[\cite{Newman}, {\cite[Thm.~4.15]{Schneider90}}]\label{theo:subalgebra-ideal}
Let $H$ be a cocommutative Hopf algebra. Then
\[\xymatrix{
\big\{I \subset H \;|\; \text{$I$ left ideal two-sided coideal}\big\}
\ar[r]<2pt>^-{\varphi} & 
\big\{A \subset H \;|\; \text{$A$ Hopf subalgebra}\big\} 
\ar[l]<2pt>^-{\psi}}\]
with $\varphi(I) = {}^{\text{co}\,H/I}H$ and $\psi(A) = HA^+$ are inverse bijections.
\end{theo}

\begin{cor}[{\cite[Theorem 3.4.6]{Montgomery}}] \label{cor:subalgebra-ideal}
The bijective correspondence from the previous theorem can be restricted to a bijection between the normal Hopf subalgebras and Hopf ideals of a cocommutative Hopf algebra $H$.
\end{cor}

If $I$ is a Hopf ideal, the set of left $H/I$-coinvariants of $H$ is exactly the kernel of the canonical surjection $H\to H/I$ in the category of Hopf algebras.

The following follows easily by a standard computation. 

\begin{lem}\label{lem:sub-basics}
Let $H$ be a Hopf-algebra.
\begin{enumerate}[(a)]
\item Let $I_1, I_2$ be left ideals two-sided coideals in~$H$. Then $I_1 + I_2$ is also a left ideal two-sided coideal in~$H$.
\item Let $I_1, I_2$ be Hopf ideals in~$H$. Then $I_1 + I_2$ is also a Hopf ideal in~$H$.
\item Let $H_1, H_2$ be Hopf-subalgebras of~$H$. Then $H_1 \cap H_2$ is also a Hopf-subalgebra of~$H$.
\item Let $H_1, H_2$ be normal Hopf-subalgebras of~$H$. Then $H_1 \cap H_2$ is also a normal Hopf-subalgebra of~$H$.
\end{enumerate}
\end{lem}

Recall that for a bialgebra $H$, the category ${}_H\Mod$ of (left) $H$-modules is monoidal, where the tensor product of two left $H$-modules $M$ and $N$ is endowed with a left $H$-module structure via the action
$$h\cdot (m\ot n)=h_{(1)}\cdot m\ot h_{(2)}\cdot n$$
for all $h\in H$, $m\in M$ and $n\in N$.
An algebra object $L$ in the category ${}_H\Mod$ is then called a (left) $H$-module algebra, which means that it is an algebra $L$ that is at the same time a left $H$-module such that the following compatibility conditions hold
$$h\cdot (xy)=(h_{(1)}\cdot x)(h_{(2)}\cdot y); \qquad h\cdot 1_L=\epsilon(h)1_L,$$
for all $h\in H$ and $x,y \in L$. Remark that if $h$ is a grouplike element in $H$, then the action of $h$ on $L$ gives an algebra automorphism. 

Similarly, a coalgebra object in the category  ${}_H\Mod$ is called a (left) $H$-module coalgebra. This is a coalgebra $C$ that is at the same time a left $H$-module such that the following compatibility conditions hold
\[\Delta(h\cdot c)=h_{(1)}\cdot c_{(1)}\ot h_{(2)}\cdot c_{(2)}; \qquad \epsilon_C(h\cdot c)=\epsilon_H(h)\epsilon_C(c),\]
for all $h\in H$ and $c\in C$. Again, if $h$ is a grouplike element in $H$, then $h$ acts on $C$ by a coalgebra automorphism.

For a left $H$-module algebra $L$, we can consider the smash product algebra $L\# H$, which is the vector space $H\ot L$ endowed with the multiplication
$$(x\ot h)(y\ot k)=x(h_{(1)}\cdot y) \ot h_{(2)}k.$$
Then a (left) $L\# H$-module $M$ is at the same time a left $L$-module and a left $H$-module satisfying the following compatibility condition
$$h\cdot (xm)= (h_{(1)}\cdot x)(h_{(2)}\cdot m)$$
for all $h\in H$, $x\in L$ and $m\in M$.
In our current setting, the famous {\em faithfully flat descent} for Hopf-Galois extensions can be expressed in the following way.

\begin{proposition}\label{pr:descent}
If $L/K$ is a finite Hopf-Galois extension, then the functors
$$\xymatrix{L\ot -:{\sf Vect}_K \ar@<.5ex>[r] 
& {_{L\# H}{\sf Mod}} :(-)^{H} \ar@<.5ex>[l] }$$
define an equivalence of categories, where $M^H=\{m\in M ~|~ h\cdot m=\epsilon(h)m, \forall h\in H\}$.

In case $H=K[G]$ is a group algebra, the above equivalence is even an equivalence of symmetric monoidal categories, and $M^H$ coincides with $M^G$, the set of $G$-invariants.
\end{proposition}

\begin{lemma}\label{grouplikes}
Let $A$ be a finite dimensional vector space over a field $K$, and let $\{(e_i,f_i)~|~i=1,\ldots,n\}\subset A\times A^*$ be a finite dual base.
\begin{enumerate}[(a)]
\item \label{grouplikes:a} There is a natural bijective correspondence between the algebra structures on $A$ and the coalgebra structures on $A^*$. Explicitly,
the correspondence between a multiplication for $A$ and a comultiplication for $A^*$ is given by the following formulas for all $a,b\in A$ and $f\in A^*$
\begin{equation}\label{multcomult}
\Delta(f)=\sum_{i,j} f(e_ie_j) f_i\ot f_j\qquad ab=\sum_i f_{i(1)}(a)f_{i(2)}(b)e_i .
\end{equation}
Otherwise stated
$$f(ab)=f_{(1)}(a)f_{(2)}(b)$$
for all $a,b\in A$ and $f\in A^*$.
\item \label{grouplikes:b} Considering an algebra structure on $A$ and a corresponding coalgebra structure on $A^*$ as in \eqref{grouplikes:a}, the set of grouplike elements in $A^*$ is exactly the set of algebra morphisms from $A$ to $K$. Moreover, the coalgebra $A^*$ has a base of grouplike elements if and only if the algebra $A$ has a base of orthogonal idempotents (i.e.\ $A$ is isomorphic to a product of copies of the base field).
\item \label{grouplikes:c} Let $L/K$ be any field extension and fix a $K$-algebra structure on $A$. Then $\Hom_K(A,L)$ is an $L$-coalgebra, whose grouplike elements are exactly the set of $K$-algebra morphisms $\Hom_{K-\Alg}(A,L)$.
\end{enumerate}
\end{lemma}

\begin{proof}
\eqref{grouplikes:a} This well-known result follows by a direct computation.

\eqref{grouplikes:b} Using the left hand side of \eqref{multcomult}, one sees that $f$ is multiplicative exactly if and only if $\Delta(f)=f\ot f$. Applying the right hand side of \eqref{multcomult} on base elements $e_i$, one sees that the dual base elements $f_i$ are grouplike if and only if the base elements are orthogonal idempotents.

\eqref{grouplikes:c} If $A$ is a finite dimensional $K$-algebra, then by base extension we find that $L\ot_K A$ is a finite dimensional $L$-algebra, so applying the previous parts we find that the $L$-linear dual $\Hom_L(L\ot_ KA,L)\cong \Hom_K(A,L)$ is an $L$-coalgebra whose grouplikes are exactly the $L$-algebra morphisms in $\Hom_L(L\ot_ KA,L)$, which correspond to the $K$-algebra morphisms in $\Hom_K(A,L)$ using the last isomorphism.
\end{proof}

\begin{proposition}\label{subquot}
\begin{enumerate}[(a)]
\item \label{subquot:a}
If $C$ is a coalgebra with a base $B$ of grouplike elements, then any quotient coalgebra and any subcoalgebra of $C$ also has a base of grouplike elements (formed by subsets of $B$).
\item \label{subquot:b}
Any Hopf subalgebra of a group algebra $K[G]$ is again a group algebra $K[N]$ over a subgroup $N$ of $G$. Moreover $K[N]$ is a normal Hopf subalgebra of $K[G]$ if and only if $N$ is a normal subgroup of $G$. 
\item \label{subquot:c}
If $I=K[G]K[N]^+$ is the left ideal two-sided coideal of $K[G]$ associated to the Hopf subalgebra $K[N]$, then the quotient $K[G]$-module coalgebra $K[G]/I$ is isomorphic to $K[G/N]$ (the coalgebra with right cosets of $N$ as a base). In particular any quotient Hopf algebra of a group algebra is again a group algebra over a quotient group (this happens when when $K[N]$ is a normal Hopf subalgebra).
\item \label{subquot:d}
The kernel of the projection $K[G]\to K[G/N]$ is the linear span of elements of the form $gn-gn'$ where $g\in G$ and $n,n'\in N$.
\end{enumerate}
\label{grouplikesHopf}
\end{proposition}

\begin{proof}
\eqref{subquot:a}
If $C$ has a base $B$ of grouplike elements and $f:C\to D$ is a surjective coalgebra morphism, then $f(B)$ is a generating set of grouplike elements. Since grouplike elements are always linearly independent, this gives a base of grouplike elements for~$D$.

If $i:D\to C$ is an injective coalgebra morphism, then $i^*:C^*\to D^*$ is a surjective algebra morphism. From Lemma~\ref{grouplikes} we know that $C^*$ is a product of copies of the base field. Consequently, $D^*$ is also a product of (a smaller number of) copies of the base field, and therefore $D$ has a base of grouplike elements. The inclusion map $i$ sends grouplike elements of $D$ to grouplike elements of $C$, so the base of grouplikes of $D$ can be considered as a subset of the base of grouplikes of~$C$.

\eqref{subquot:b}
This follows from part \eqref{subquot:a}, using the well-known and easily checked fact that the set of grouplike elements in a Hopf algebra forms a group.

\eqref{subquot:c}-\eqref{subquot:d}
First remark that $I$ is generated as a $K$-linear space by elements of the form $gn-gn'$ where $g\in G$ and $n,n'\in N$.
By part \eqref{subquot:a}, we know that $K[G]/I$ has a base of grouplike elements, which can be obtained by taking the images of the elements of $G$ under the projection $\pi:K[G]\to K[G]/I$. Suppose that $g,g'\in G$ are such that $\pi(g)=\pi(g')$ in $K[G]/I$.
By the above characterizations of elements in $I$ and the linear independence of the grouplike elements in $K[G]$, this is equivalent with the existence of elements $h\in G$ and $n,n'\in N$ such that $g=hn$ and $g'=hn'$. Otherwise said, $g^{-1}g'=n^{-1}n\in N$ or $g'\in gN$. Hence we conclude that $G/N$ is indeed a base for $K[G]/I$. The other results follow directly from this.
\end{proof}

\subsection{Basics of Hopf-Galois extensions}\selabel{basics}

We start this section by explicitly specialising Definition~\ref{definfiniteHG} to the finite dimensional case.

\begin{defi}\label{defiHGal}
Let $L/K$ be a finite field extension and let $H$ be a finite dimensional $K$-bialgebra.
We say that $L/K$ is a  {\em Hopf-Galois extension} for $H$ (or simply {\em $H$-Galois}) if $L$ is a left $H$-module algebra and the $K$-linear map
\begin{equation}\label{can}
\xymatrix{\can : L \ot H \ar[r] & \End_K(L),\  \can(x \o h)(y)=x(h\cdot y)}
\end{equation}
where $x,y\in L$ and $h\in H$, is bijective. The map $\can$ is called the {\em Galois} map or the {\em canonical} map.
\end{defi}

We first prove that $H$ as in Definition~\ref{defiHGal} is a cocommutative $K$-Hopf algebra. Moreover, by faithfully flat descent, it follows that the space of $H$-invariants of $L$ is just $K$, which means that $L$ is even a {\em Galois object} instead of just a Galois extension.

\begin{lem}\label{lem:cocommutativeHopf}
Let $H$ be a $K$-bialgebra.
If $L/K$ is a finite extension that is $H$-Galois, then $H$ is a cocommutative Hopf $K$-algebra.
Moreover, the $H$-invariants of $L$ are just $K$, that is $K=L^H=\{x\in L~|~h\cdot x=\epsilon(h)x, \forall h\in H\}$.
\end{lem}

\begin{proof}
Using the fact that $L/K$ is $H$-Galois in the first and third isomorphism, the fact that $L/K$ is finite
in the second isomorphism and the Hom-tensor relations in the last isomorphism we obtain a natural isomorphism
\begin{eqnarray*}
L \ot H \ot H  \cong& \Hom_K(L,L) \ot H &\cong \Hom_K(L, L \ot H)\\
 \cong &\Hom_K(L,\Hom_K(L,L)) &\cong \Hom_K(L \ot L, L)
\end{eqnarray*}
The composed isomorphism $\alpha:L \ot H \ot H \to \Hom_K(L \ot L, L)$ is given explicitly by ${\alpha(x \ot h \ot h')(y \ot z)}=x (h\cdot y) (h'\cdot z)$.

By the commutativity of $L$, it is clear that for all $x,y,z\in L$ and all $h\in H$,
\begin{eqnarray*}
x(h_{(1)}\cdot y)(h_{(2)}\cdot z)&=& x h\cdot (yz) = x h\cdot (zy)\\
&=& x(h_{(1)}\cdot z)(h_{(2)}\cdot y)= x(h_{(2)}\cdot y)(h_{(1)}\cdot z)
\end{eqnarray*}
This means that $\alpha(x\ot h_{(1)}\ot h_{(2)})=\alpha(x\ot h_{(2)}\ot h_{(1)})$ and since $\alpha$ is an isomorphism we also have that $x\ot h_{(1)}\ot h_{(2)})=x\ot h_{(2)}\ot h_{(1)}$. Since $K$ is a field,
it follows that $h_{(1)}\ot h_{(2)}=h_{(2)}\ot h_{(1)}\in H\ot H$, hence $H$ is cocommutative.

The fact that $H$ is a Hopf algebra was proven in \cite{Peter97} (there with $H$-comodule algebras rather than $H$-module algebras). A brief argument goes as follows. Consider the map
\begin{eqnarray*}
\beta:&& L \ot H \ot H \to \Hom_K(L \ot L, L),\\ &&\beta(x\ot h\ot h')(y\ot z)=x(h\cdot (y(h'\cdot z))=x(h_{(1)}\cdot y)(h_{(2)}h'\cdot z).
\end{eqnarray*}
By a similar reasoning as for the map $\alpha$ above, one can see that $\beta$, being a composition of natural isomorphisms, is itself an isomorphism. Moreover, one easily observes that $\beta=\alpha\circ (id_L\ot \can')$ where $\can':H\ot H\to H\ot H$ is given by $\can'(h\ot h')=h_{(1)}\ot h_{(2)}h'$. Since both $\alpha$ and $\beta$ are isomorphisms, $\can'$ is an isomorphism as well. Then one can verify that $S:H\to H$ given by $S(h)=(\epsilon\ot id_H)\circ \can'^{-1}(h\ot 1)$ is an antipode for $H$.
%
%

For the last statement, remark that $L\cong L\ot K$ is a left $L\# H$-module by means of the usual actions of $L$ and $H$ on $L$. Hence, by \prref{descent} we immediately obtain that $L^H$ and $K$ coincide as subsets of $L$.
\end{proof}

Let now $L$ be any $H$-module algebra over~$K$ (not necessarily $H$-Galois) and consider any field extension $L\subset \tilde L$ and the action of $\tilde L\ot H$ on $\Hom_K(L,\tilde L)$ defined as follows: For $y\otimes h \in \tilde L\ot H$ and $f \in \Hom_K(L,\tilde L)$, define $f.(y \otimes h)$ as the map sending $x \in L$ to $y  f(h \cdot x)$. This turns $\Hom_K(L,\tilde L)$ into a right $\tilde L\ot H$-module. Combining this action with the $\tilde L$-coalgebra structure on $\Hom_K(L,\tilde L)$ as described in Lemma \ref{grouplikes}\eqref{grouplikes:c}, one arrives at the following result.

\begin{lemma}\label{cancoalgmap}
With structure as defined above, $\Hom_K(L,\tilde L)$ is a right $\tilde L\ot H$-module $\tilde L$-coalgebra and the canonical map $\widetilde\can: \tilde{L} \otimes_K H \to \Hom_K(L,\tilde{L})$ is a morphism of right $\tilde L\ot H$-module $\tilde L$-coalgebras.
\end{lemma}

\begin{proof}
Remark that the $\tilde L\ot H$-action on $\Hom_K(L,\tilde L)$ is exactly such that
$$\widetilde\can(y\ot h)=\id_L.(y\ot h).$$
From this observation it is clear that $\widetilde\can$ is already $\tilde L\ot H$-linear.
We check that for all $y\ot h\in \tilde L\ot H$ and $x,x'\in L$ :
\begin{eqnarray*}
(\widetilde\can(y\ot h_{(1)})\ot \widetilde\can(1\ot h_{(2)}))(x\ot x')&=&(\widetilde\can(y\ot h_{(1)})(x)\widetilde\can(1\ot h_{(2)}))(x')\\
&=&y (h_{(1)} \cdot x)(h_{(2)} \cdot x')\\ &=& y (h\cdot (xx')) =\widetilde\can(y\ot h)(xx')
\end{eqnarray*}
Hence $\Delta(\widetilde\can(y\ot h))=\widetilde\can(y\ot h_{(1)})\ot \widetilde\can(1\ot h_{(2)})$, so $\widetilde\can$ is indeed a coalgebra morphism.
\end{proof}

Notice that the $\tilde L$-module coalgebra structure on $\Hom_K(L,\tilde L$) is exactly dual to the $\tilde L\ot H$-module algebra structure on $\tilde L\ot L$. Via the action, we obtain the $\tilde{L}$-linear homomorphism
\[\tilde{L} \otimes_K H \to \End_{\tilde{L}}\big(\Hom_K(L,\tilde{L}) \big), \;\;\; y \; h \mapsto (f \mapsto (x \mapsto y  f(h \cdot x))).\]
If we compose it with the evaluation at $\id_L: L \to \tilde{L}$, we exactly recover the canonical map $\widetilde\can: \tilde{L} \otimes_K H \to \Hom_K(L,\tilde{L})$.

Let us briefly recall the dual point of view to Hopf-Galois theory, which is in fact more often used in literature.
If $H$ is a finite dimensional Hopf $K$-algebra, then the dual space $H^*=\Hom_K(H,K)$ is again a finite dimensional Hopf algebra, whose multiplication and comultiplication are obtained by dualizing those of $H$, as in \eqref{multcomult}. Similarly, $L$ is a left $H$-module algebra by an action $H\ot L\to L, h\ot x\mapsto h\cdot x$ if and only if $L$ is a right $H^*$-comodule algebra by the coaction
\begin{equation}\label{coaction}
\rho: L\to L\ot H^*,\ \rho(x)=x_{[0]}\ot x_{[1]}= \sum_i (e_i\cdot x)\ot f_i,
\end{equation}
where $\{(e_i,f_i)\}$ is a finite dual base for $H$. 
The (right) $L$-dual of the Galois map  leads to a second canonical map
\begin{equation}\label{can*} 
\can^*:L\ot L\to L\ot H^*,\ \can(x\ot y)=xy_{[0]}\ot y_{[1]}= \sum_i x (e_i\cdot y)\ot f_i.
\end{equation}
The map $\can$ from \eqref{can} is an isomorphism (i.e. $L/K$ is $H$-Galois) if and only if $\can^*$ is an isomorphism. Moreover 
the map $\can^*$ can be checked to be an algebra morphism (thanks to the commutativity of $L$), which the dual statement of Lemma~\ref{cancoalgmap}. 
Remark that since $H$ is cocommutative, $H^*$ is a commutative Hopf algebra.

\subsection{The correspondence map $\Phi$}

In this subsection we study $J(L_0)$ for an intermediate extension~$L_0$.

\begin{lem}\label{lem:surj}
Let $L/K$ be $H$-Galois and $L_0$ be an intermediate field. Then the map
$$\can_0:L \otimes H/J(L_0) \to \Hom_K(L_0,L),\ \can_0(x\ot \ol h)(y)=x(h\cdot y),$$ induced by the Galois map, is surjective. If $L_0$ is $H$-stable, then $\can_0$ induces a well-defined map $\can_0':L_0 \otimes H/J(L_0) \to \End_K(L_0)$, which is also surjective.
\end{lem}

\begin{proof}
The first statement directly follows from the following commutative diagram, where the vertical arrows are the obvious surjections,
\[\xymatrix{L \ot H \ar[r]^-{\can}_-\sim\ar@{->>}[d]& \End_K(L)\ar@{->>}[d]\\ L \ot H/J(L_0)\ar[r]^{\can_0} & \Hom_K(L_0, L).}\]
If $L_0$ is $H$-stable, then the map $L_0 \otimes H/J(L_0) \to \End_K(L_0)$ is clearly well-defined. To see that it is surjective, consider first the following surjective map:
\[ L \ot_{L_0} L_0 \ot H/J(L_0) \cong L \ot H/J(L_0) \twoheadrightarrow \Hom_K(L_0, L) \cong L \ot_{L_0} \End_K(L_0).\]
As $L_0 \to L$ is faithfully flat, the surjectivity of $L_0 \ot H/J(L_0) \to \End_K(L_0)$ follows.
\end{proof}

We now formulate some equivalent characterizations of $H$-subextensions.

\begin{lem}\label{lem:Hsubext}
Let $L/K$ be $H$-Galois and $L_0$ be an intermediate field. Denote $\alpha_0:H\to \Hom_K(L_0,L), \alpha_0(h)(x)=h.x$ for any $x\in L_0$ and $h\in H$.
Then the following statements are equivalent:
\begin{enumerate}[(i)]
\item \label{lem:Hsubext:i} For any subset $F \subset H$, $K$-linear independence inside $\Hom_K(L_0,L)$ implies $L$-linear independence.
\item \label{lem:Hsubext:ii} There are elements $h_1,\dots,h_m \in H$ whose images under $\alpha_0$ are $L$-linearly independent and generate $\alpha_0(H)$ over $K$.
\item \label{lem:Hsubext:iii} The natural map $\can_0 : L \otimes H/J(L_0) \to \Hom_K(L_0,L)$ is injective.
\item \label{lem:Hsubext:iv} $L_0$ is an $H$-subextension of $L/K$.
\end{enumerate}
\end{lem}

\begin{proof}
Suppose \eqref{lem:Hsubext:i}. Let $B$ be a $K$-basis of $H$, then there exists a subset $B_0 \subset B$ such that $\alpha_0(B_0)$ is a $K$-basis of $\alpha_0(H)$. By assuption, $\alpha_0(B_0)$ is also $L$-linearly independent in $\Hom_K(L_0, L)$.\\
Suppose \eqref{lem:Hsubext:ii}. By definition, $J(L_0)=\ker(\alpha_0)$, hence, $\alpha_0(H)\cong H/J(L_0)$. Therefore, any element $u$ in $L \otimes H/J(L_0)$ can be written in the form $u=\sum_{i=1}^m x_i\ot \alpha_0(h_i)$ for some $x_i\in L$. Since the elements $\alpha_0(h_i)$ are $L$-linearly independent, if $\can_0(u)=\sum_{i=1}^m x_i\alpha_0(h_i)=0$, then all $x_i=0$ and therefore $u=0$. Hence $\can_0$ is injective. \\
Suppose \eqref{lem:Hsubext:iii}. If $\can_0$ is injective, it is also bijective by Lemma~\ref{lem:surj}).\\
Suppose \eqref{lem:Hsubext:iv}. 
Let $F \subset H$ such that $\alpha_0(F)$ is $K$-linearly independent. Since $\alpha_0$ has an epi-mono factorization $\alpha:\xymatrix{H \ar@{->>}[r]^-\pi & H/J(H_0)  \ar@{^(->}[r]^-{\alpha'_0}& \Hom_K(L_0,L)}$, we find that $\pi(F)$ is also $K$-linearly independent. Now suppose that $\sum x_i\alpha_0(f_i)=0$ for some $x_i\in L$ and $f_i\in F$. Since $\sum x_i\alpha_0(f_i)=\can_0(x_i\ot \pi(f_i))$ and $\can_0$ is injective, we find that $x_i\ot \pi(f_i)=0$ and hence all $x_i=0$ as the elements $\pi(f_i)$ are $K$-linearly independent. It follows that $\alpha_0(F)$ is $L$-linearly independent as needed.
\end{proof}

\begin{proposition}\label{prop:HopfIdeal}
Let $L/K$ be $H$-Galois and $L_0$ be an intermediate extension.
\begin{enumerate}[(a)]
\item \label{prop:HopfIdeal:aa} $J(L_0)$ is a left ideal and $\epsilon(J(L_0))=0$.
\item \label{prop:HopfIdeal:a} If $L_0$ is an $H$-subextension, then $J(L_0)$ is a two-sided coideal.
\item \label{prop:HopfIdeal:b} If $L_0$ is $H$-stable, then $J(L_0)$ is a right ideal.
\item \label{prop:HopfIdeal:c} If $L_0$ is $H$-normal, then $L_0/K$ is $H/J(L_0)$-Galois and $J(L_0)$ is a Hopf ideal.
\end{enumerate}
Consequently, by defining $\Phi(L_0)=J(L_0)$, we obtain the map needed in Theorem~\ref{theo:fin-cor}.
\end{proposition}

\begin{proof}
\eqref{prop:HopfIdeal:aa}
Since for any $h\in H$, $h'\in J(L_0)$ and $x\in L_0$ we have that
\[(hh')\cdot x=h\cdot (h'\cdot x)=0\]
and $J(L_0)$ is indeed a left ideal. Moreover, $\epsilon(J(L_0))=0$ since $1\in L_0$ and hence for any $h\in J(L_0)$, $h\cdot 1=\epsilon(h)1=0$, so $\epsilon(h)=0$.

\eqref{prop:HopfIdeal:a}
We need to prove that $\Delta(J(L_0)) \subset H \ot J(L_0) + J(L_0) \ot H$. Denote $\pi : H \to H/J(L_0) : h \mapsto \bar{h}$ the natural projection.
Then $H \ot J(L_0) + J(L_0) \ot H=\ker(\pi\ot\pi)$. Hence for any $h\in J(L_0)$ we have that $\Delta(h)\in H \ot J(L_0) + J(L_0) \ot H$ if and only if $\pi\ot \pi\circ \Delta(h)=0$. To prove this, let $\big\{\bar{h}_1, \dots, \bar{h}_n\big\}$ be a basis of $H/J(L_0)$, and take elements $\bar{h}_1', \dots, \bar{h}_n' \in H/J(L_0)$ such that $(\pi \ot \pi)\Delta(h) = \bar{h}_{(1)} \ot \bar{h}_{(2)} = \sum_{i=1}^n \bar{h}_i' \ot \bar{h}_i$. For any $x, y\in L_0$ we find
\[\can_0\big(\sum_{i=1}^n \bar{h}_i' \cdot x \ot \bar{h}_i\big)(y) = \sum_{i=1}^n(\bar{h}_i' \cdot x)(\bar{h}_i \cdot y) = h\cdot (xy) = 0.\]
Since $\can_0$ is injective by the definition of $H$-subextensions, we obtain that $\sum_{i=1}^n \bar{h}_i' \cdot x \ot \bar{h}_i = 0$.
Because the elements $\bar h_i$ are a base, it follows that  $\bar{h}_i' \cdot x = 0$ $\forall x \in L_0$, hence $\bar{h}'_i = 0$ for all indices $i$ and we can conclude that $\pi\ot \pi\circ \Delta(h)=0$.\\
\eqref{prop:HopfIdeal:b}
Let $h \in J(L_0)$ and $h' \in H$, then $hh' \in J(L_0)$ because $\forall x \in L_0 : (hh')\cdot x = h(h' \cdot x) = 0$ (since $h' \cdot x \in L_0$).\\
\eqref{prop:HopfIdeal:c}
The Galois map $\can'_0:L_0 \ot H/J(L_0) \to \End_K(L_0)$ is surjective by Lemma~\ref{lem:surj}. Since $\can'_0$ is the restriction of the map $\can_0$, which is injective by the definition of $H$-subextensions, it is injective as well, hence $L_0/K$ is $H/J(L_0)$-Galois. Consequently, $J(L_0)$ is a Hopf ideal because $H/J(L_0)$ is a Hopf algebra by Lemma \ref{lem:cocommutativeHopf}.
Obviously, the map $\Phi$ is inclusion reversing.
\end{proof}

\begin{cor}\label{cor:Phi'}
Let $L/K$ be $H$-Galois and denote by $\alpha:H\to \End_K(L)$ the map associated with the action of $H$ on $L$. Then for any intermediate extension $L_0$, the following subsets of $H$ coincide:
\[\{h\in H~|~h_{(1)}\cdot x\ot h_{(2)} = x\ot h, \forall x\in L_0 \}=\{h \in H \;|\; \alpha(h) \in \End_{L_0}(L)\}. \]
Denoting the above subset of $H$ by $\Phi'(L_0)$, we obtain inclusion reversing maps
$$\Phi':\big\{L/L_0/K \st L_0 \t{ $H$-subextension}\big\} \longrightarrow \big\{H' \subset H \st H' \t{ Hopf subalgebra}\big\}$$
and
$$\Phi':\big\{L/L_0/K \st L_0 \t{ $H$-normal}\big\} \longrightarrow \big\{H' \subset H \st H' \t{ normal Hopf subalgebra}\big\}.$$
\end{cor}

\begin{proof}
Consider $h\in H$ such that $h_{(1)}\cdot x_0\ot h_{(2)}=x_0\ot h\in L\ot H$, for all $x_0\in L_0$.
When we apply the bijective map $\can$ to this equality, we find that the previous equality holds if and only if 
$$(h_{(1)}\cdot x_0)(h_{(2)}\cdot x) = x_0 (h\cdot x)$$
for all $x_0\in L_0$ and all $x\in L$. As $(h_{(1)}\cdot x_0)(h_{(2)}\cdot x)=h\cdot (x_0x)$, we conclude that $h\in H$ satisfies the above equality if and only if $\alpha(h)$ is $L_0$-linear.

The fact that $\Phi$ is a well-defined and inclusion reversing map is obtained by remarking that $\Phi'=\phi\circ \Phi$, where $\phi$ is the correspondence from Theorem~\ref{theo:subalgebra-ideal} and Corollary~\ref{cor:subalgebra-ideal}. Indeed, since $\alpha_0:H/J(L_0)\to \Hom_K(L_0,L)$ is injective, we have that $h\in {}^{{\rm co} H/J(L_0)}H=\phi\circ \Phi(L_0)$ if and only if $h\in \Phi'(L_0)$.
\end{proof}

\subsection{The space of invariants}\selabel{invariant}

In this subsection, we define and study the space of invariants for an $H$-module algebra $L$.
This allows us to define the map $\Psi$ from Theorem~\ref{theo:fin-cor}.

\begin{lemma}\label{prop:intermediate}
Let $L$ be an $H$-module algebra and $I \subset H$. Denote as before $L^I$ the space of $I$-invariants.
\begin{enumerate}[(a)]
\item \label{prop:intermediate:a} If $I$ is a left ideal two-sided coideal, then $L^I$ is an intermediate extension of $L/K$.
\item \label{prop:intermediate:b} If $I$ is moreover a right ideal (hence a bi-ideal), then $L^I$ is $H$-stable.
\end{enumerate}
\end{lemma}

\begin{proof}
\eqref{prop:intermediate:a} Let $h \in I$ and $x\in K$.
Then $h\cdot x = \epsilon(h)x = 0$, so $L^I$ contains $K$. 
For any $x,y\in L^I$, we find that $h\cdot (xy) = (h_{(1)}\cdot x)(h_{(2)}\cdot y) = 0$, 
since $\Delta(h)\in I\ot H+H\ot I$. 
Hence $L^I$ is a ring and it is a field because $L/K$ is an algebraic extension.\\ 
\eqref{prop:intermediate:b} If $I$ is also a right ideal, then $\forall h' \in H, \forall x \in L^I: h\cdot (h'\cdot x) = (hh')\cdot x = 0$ because $hh' \in I$. So $h'\cdot x \in L^I$.
\end{proof}

\begin{lem}\label{lem:extend-rationality}
Let $L/K$ be an $H$-module algebra and let $I \subseteq H$ be a left ideal two-sided coideal. Let $\tilde{L}/L$ be any field extension.
Then $\can':L\ot H/J(L^I) \to \Hom_K(L^I,L)$ is injective if one of the following equivalent statements holds:
\begin{enumerate}[(i)]
\item \label{lem:extend-rationality:i} The natural map $L \otimes_K H/I \to \Hom_K(L^I,L)$ is injective.
\item \label{lem:extend-rationality:ii} The natural map $\tilde{L} \otimes_K H/I \to \Hom_K(L^I,\tilde{L})$ is injective.
\end{enumerate}
If moreover $L/K$ is $H$-Galois then the validity of any of these conditions implies that $L^I$ is an $H$-subextension and $I = J(L^I)$.
\end{lem}

\begin{proof}
The equivalence between \eqref{lem:extend-rationality:i} and \eqref{lem:extend-rationality:ii} follows from the fact that $\tilde L$ is faithfully flat over $L$. Since $I\subset J(L^I)$, we have the projection $p:H/I\to H/J(L^I)$. Hence the canonical map from the statement factors as the composition
\[\xymatrix{L\ot H/I \ar[rr]^-{id\ot p} && L\ot H/J(L^I) \ar[rr]^-{\can'} && \Hom_K(L^I,L)}\]
and it follows that $\can'$ is injective as well. 
Hence, $L^I$ is an $H$-subextension if $L/K$ is $H$-Galois.
\end{proof}

Next we observe that the notion invariants is stable under base change. Recall that if $L$ is an $H$-module algebra, and $L\subset \tilde L$ is any field extension, then $\tilde L\ot L$ is an $\tilde L\ot H$-module algebra.

\begin{lemma}\label{invariantsbasechange}
Let $L$ be an $H$-module $K$-algebra and $F\subset H$ a $K$-linear subspace. Then for any base extension $\tilde L/L$, we have that $$(\tilde L\ot L)^{\tilde L\ot F} = \tilde L\ot L^{F}.$$
\end{lemma}

\begin{proof}
This follows from the fact that the space of invariants is a limit in $\Vect_K$ and the extension-of-scalars functor preserves limits. An explicit argument is as follows.
Take $\tilde x\ot x\in \tilde L\ot L^{F}$. Then for any $\tilde y\ot h\in L\ot F$, we find that 
$$(\tilde y\ot h)\cdot (\tilde x \ot x)=\tilde y\tilde x\ot h\cdot x=\tilde y\tilde x\ot \epsilon(y)x=\epsilon(\tilde y\ot h)\tilde x\ot x.$$
So $\tilde L\ot L^{F}\subset (\tilde L\ot L)^{\tilde L\ot F}$. On the other hand, take any $\sum_i\tilde x_i\ot x_i\in (\tilde L\ot L)^{\tilde L\ot F}$, where we suppose without loss of generality that the elements $\tilde x_i$ are linearly independent. Then for any $h\in F$, we find using the definition of the action under extension of scalars that
$$(1\ot h)\cdot (\sum_i\tilde x_i\ot x_i) = \sum_i\tilde x_i\ot h\cdot x_i$$
and on the other hand, since $\sum_i\tilde x_i\ot x_i$ is $F$-invariant, we find
$$(1\ot h)\cdot (\sum_i\tilde x_i\ot x_i) = \sum_i\tilde x_i\ot \epsilon(h)x_i$$
Since the elements $\tilde x_i$ are linearly independent we can conclude that $x_i\in L^F$ for all indices $i$, and hence $\sum_i\tilde x_i\ot h\cdot x_i\in \tilde L\ot L^F$.
\end{proof}

Let us show that the space of invariants with respect to a Hopf subalgebra coincides with the space of invariants for the associated ideal. Hence the invariants with respect to an ideal, which we will consider mostly here, correspond in the case of usual Galois extensions to the classical invariants.

\begin{lemma}\label{le:invariants_coinvariants}
Let $L/K$ be a finite $H$-Galois extension and let $I$ be a left ideal two-sided coideal of $H$. Consider the associated Hopf subalgebra  $H_0\subset H$ as in Theorem~\ref{theo:subalgebra-ideal}.
By dualizing we find that $(H/I)^*$ is a right coideal subalgebra of $H^*$ and $\pi:H^*\to H_0^*$ is a Hopf algebra surjection. Then the following subsets of $L$ coincide:
$$L^I=\rho^{-1}(L\ot (H/I)^*)=L^{co H_0^*}=L^{H_0},$$
where $\rho:L\to L\ot H^*, \rho(x)=x_{[0]}\ot x_{[1]}$ is the coaction as in \eqref{coaction} and $L^{co H_0^*}=\{x\in L~|~x_{[0]}\ot \pi(x_{[1]})=x\ot 1_{H_0^*}\}$ is the space of coinvariants.
\end{lemma}
\begin{proof}
\ul{$L^I \subset \rho^{-1}(L \o (H/I)^*)$}. Take any $x \in L^I$ and take a finite dual base $\{(e_i,f_i)\}$ of $H$, whose first $m$ elements are in $I$ and the next $n-m$ elements generate a linear complement of $I$ in $H$. Then $f_{m+1},\ldots, f_n$ provide exactly a base of $(H/I)^*$ and we find 
$$\rho(x)= \sum_{i=1}^n (e_i\cdot x)\ot f_i = \sum_{i=m+1}^n (e_i\cdot x)\ot f_i\in L\ot A.$$
\ul{$\rho^{-1}(L \o (H/I)^*) \subset L^{\t{co}\,H_0^*}$}. 
Recall that $I=HH_0^+$. Then since for any $x\in H_0$, we have that $x-\epsilon(x)1\in H_0^+\subset I$, we find that the composition $H_0\subset H \twoheadrightarrow H/I$ is given by the map $x\mapsto \epsilon(x)1$. Dualizing this, we find that $\pi(a)=\epsilon_{H^*}(a)1_{H_0^*}$ for any $a\in (H/I)^*$. Now let $x \in L$ such that $\rho(x) \in L \o A$, then $x_{[0]} \o \pi(x_{[1]}) = x_{[0]} \o \epsilon(x_{[1]})1_{H_0^*} = x \o 1_{H_0^*}$.\\
\ul{$L^{\t{co}\,H_0^*} \subset L^{H_0}$}. Let $x \in L^{\t{co}\,H_0^*}$, then $x_{[0]} \o \pi(x_{[1]}) = x \o 1_{H_0^*}$. So $\forall b \in H_0$, $bx = x\, \big\la 1_{H_0^*}, b \big\ra = \epsilon(b)x$ and therefore $x \in L^{H_0}$.\\
\ul{$L^{H_0} \subset L^I$}. it is easy to see that $L^{H_0} \subset L^{H_0^+} = L^{HH_0^+}$. To finish the proof, just recall that $HH_0^+ = I$.
\end{proof}

\subsection{Greither-Pareigis Theory}

Let $L/K$ be a finite separable extension.
We view this extension in two ways: via classical Galois theory and via Hopf-Galois theory.
For the former, we denote by $\tilde L$ the normal closure of $L$ over~$K$ and consider the Galois groups $G=\Gal(\tilde L/K)$ and $G'=\Gal(\tilde L/L)$. The natural bijection $G/G' \to \Hom_{K-\Alg}(L,\tilde L)$ given by sending $g \in G$ to its restriction to~$L$ extends to an isomorphism of $\tilde L$-coalgebras
\[ \tilde{L}[G/G'] \to \Hom_K(L,\tilde L).\]
This isomorphism is $G$-equivariant for the left $G$-actions given by 
\[ g \cdot (\sum_{hG' \in G/G'} a_{hG'} hG') = \sum_{hG' \in G/G'} g(a_{hG'}) ghG' \quad \textnormal{and} \quad g\cdot f = g \circ f\]
for $g \in G$, $\sum_{hG' \in G/G'} a_{hG'} hG' \in \tilde{L}[G/G']$ and $f \in \Hom_K(L,\tilde L)$.

We now assume further that $L/K$ is $H$-Galois. Let us recall that then $H$ is a Hopf $K$-algebra and $L$ is a left $H$-module $K$-algebra.
We then also have a natural right action of $\tilde L \ot H$ on $\Hom_K(L,\tilde L)$ by setting
\[  \big(f \cdot (y \otimes h)\big)(x) = y f(h \cdot x)\]
for $x \in L$ and $y \otimes h \in \tilde{L} \ot H$ and the canonical map, extended to $\tilde L$, is now precisely obtained by applying this right $\tilde L \ot H$-action on the identity map $\id_L: L \to \tilde L$:
\[ \widetilde{\can}: \tilde L \ot H \to \Hom_K(L,\tilde L), \;\;\; y \ot h \mapsto \id_L \cdot (y\ot h) \]
and $\big(\id_L \cdot (y\ot h)\big)(x) = y(h \cdot x)$ for $x \in L$.
The assumption that $L/K$ is $H$-Galois now means that $\widetilde{\can}$ is an isomorphism and, in other words, that $\Hom_K(L,\tilde L)$ is exactly the right $\tilde L \otimes H$-orbit of $\id_L$.
The canonical map is hence an isomorphism of right $\tilde L \otimes H$-module $\tilde L$-coalgebras, as seen in Lemma~\ref{cancoalgmap}.

We further let $G$ act from the left on $\tilde{L} \otimes_K H$ on the left factor: $g \cdot (y \otimes h) = g(y) \otimes h$ for $g \in G$ and $y \ot h \in \tilde L \ot H$. Also, $G$ acts on the left on $\Hom_K(L,\tilde L)$ by usual composition.
Then we have the following compatibility relation between the left $G$-action and the right $\tilde L \ot H$-action on $\Hom_K(L, \tilde L)$:
\begin{equation}\label{eq:compGH}
 g \circ (f \cdot (y \ot h)) = (g\circ f) \cdot (g \cdot (y \ot h))
\end{equation}
for all $f \in \Hom_K(L, \tilde L)$, all $g \in G$ and all $y \ot h \in \tilde L \ot H$, as can be easily checked by evaluating both sides at $x \in L$, as follows:
$(g \circ(f \cdot (y \ot h)))(x) = g \big(y f(h \cdot x) \big) = g(y) g(f(h\cdot x))$ and
$\big((g\circ f)\cdot (g(y \ot h))\big)(x) = \big((g\circ f) \cdot (g(y) \ot h)\big)(x) = g(y) (g\circ f)(h \cdot x) = g(y) g(f(h\cdot x))$.
We remark explicitly that $\widetilde{\can}$ is not $G$-equivariant.

\smallskip

The following theorem is our account of crucial results from \cite{GreitherPareigis87}.

\begin{theorem} 
\ \label{th:GP} 
Let $L/K$ be a finite field extension that is $H$-Galois and use notations as above. The following assertions hold.
\begin{enumerate}[(a)]
\item \label{GP:a} There is an isomorphism of Hopf $\tilde L$-algebras
\[ \tilde L\ot H\cong \tilde L[N]\]
where $N$ is a group of order $[L:K]$, if and only if $L/K$ is separable. In all further points we assume this is the case.

\item \label{GP:b} The left action of $G$ on $\tilde L\ot H$ via $g\cdot (x\ot h)=g(x)\ot h$ induces a left action of $G$ on $N$ via group automorphisms.

\item \label{GP:c} The elements of $N$ act on the right by $\tilde L$-coalgebra automorphisms on $\Hom_K(L,\tilde L)$, in particular, $N$ acts on the right on the set $\Hom_{K-\Alg}(L,\tilde L) \cong G/G'$.

\item \label{GP:d} Restricting the canonical map $\widetilde{\can}:\tilde L\ot H\to \Hom_K(L,\tilde L)$ to $N\subset \tilde L[N]\cong \tilde L\ot H$, leads to a right $N$-equivariant bijection
\[ \beta:N\to \Hom_{K-\Alg}(L,\tilde L)\cong G/G',\]
which is given by $n\mapsto (1G') \cdot n$.
In other words, $N$ acts strictly transitively from the right on $G/G'$ and $\beta$ is the orbit map of $1G'$.

\item \label{GP:e} The left $G$-action and the right $N$-action on $\Hom_{K-\Alg}(L,\tilde L)\cong G/G'$ are related by the formula
\[ g \cdot \big( (hG') \cdot n \big) = (ghG') \cdot (g\cdot n) \]
where $n \in N$ and $g,h \in G$.

\item \label{GP:g} Viewing $N$ as a subgroup of $\Perm(G/G')$ via the right $N$-action on $G/G'$ and $G$ as a subgroup of $\Perm(G/G')$ via the left $G$-action on $G/G'$, the formula of \eqref{GP:e} expresses exactly the fact that $N$ is normalized by $G$ in $\Perm(G/G')$.
\end{enumerate}
\end{theorem}

\begin{proof}
\eqref{GP:a}
The separability of $L/K$ means that the cardinality of $\Hom_{K-\Alg}(L,\tilde L)$ equals $[L:K]$, which is the $\tilde L$-dimension of $\Hom_{K}(L,\tilde L)$.
This, in turn, is by Lemma~\ref{grouplikes} equivalent to the fact that the $\tilde L$-coalgebra $\Hom_{K}(L,\tilde L)$ has a basis of grouplike elements.
Since the canonical map $\widetilde \can$ is a morphism of coalgebras and $\tilde L\ot H$ is moreover a Hopf algebra, this base of grouplike elements form a group~$N$, and hence, as a Hopf algebra, $\tilde L\ot H$ is isomorphic to a group algebra $\tilde L[N]$ for some group of order $[L:K]$.

\eqref{GP:b}
In fact the group $G$ acts on $\tilde L\otimes_K H$ by Hopf algebra automorphisms\footnote{To be more precise, $G$ acts by "relative Hopf algebra morphisms", meaning that the elements of $G$ also act as automorphisms of the base field and do not restrict to the identity on the base field.
This however, does not affect our conclusions.}. For any element $y\otimes h\in \tilde L\otimes H$ (sum understood), and any $g\in G$, we find that
\begin{multline*}
\Delta(g\cdot (y \otimes h))=\Delta(g(y)\otimes h)=g(y)\otimes h_{1}\otimes_{\tilde L} 1\otimes h_{(2)}=\\
g(y)\otimes h_{1}\otimes_{\tilde L} g(1)\otimes h_{(2)}=g\cdot (y\otimes h_{1})\otimes_{\tilde L} g\cdot (1\otimes h_{(2)})
\end{multline*}
So if $y\otimes h\in \tilde L\otimes H$ is grouplike, then also $g\cdot (y\ot h)$ will be grouplike. Since $N$ is exactly the group of grouplike elements in $\tilde L\otimes H$, we find indeed that $G$ acts on $N$. Since moreover the action of $G$ on $\tilde L\ot H$ respects the multiplication, $G$ acts moreover on $N$ by group automorphisms.

\eqref{GP:c}
Since $\Hom_K(L,\tilde L)$ is a right $\tilde L\ot H$-module $\tilde L$-coalgebra, the elements of $N$, which are exactly the grouplike elements of $\tilde L\ot H$, act from the right on $\Hom_K(L,\tilde L)$ by coalgebra automorphisms. As a coalgebra morphism sends grouplikes to grouplikes, it follows that $N$ permutes the grouplike elements of $\Hom_K(L,\tilde L)$. Finally, by Lemma \ref{grouplikes}\eqref{grouplikes:c}, the grouplike elements of $\Hom_K(L,\tilde L)$ are exactly the algebra morphism (or field morphisms) from $L$ to $\tilde L$.

\eqref{GP:d}
Since the canonical map is an $\tilde L$-coalgebra morphism by Lemma \ref{cancoalgmap}, it leads to a bijection between the set of grouplike elements of $\tilde L\ot H$, which is $N$, and the set of group\-like elements of $\Hom_L(L,\tilde L)$, which is $\Hom_{K-\Alg}(L,\tilde L)\cong G/G'$. Furthermore, again by Lemma~\ref{cancoalgmap}, $\widetilde\can$ is right $\tilde L\ot H$-linear, hence $N$-equivariant. Finally, the explicit form of this bijection in terms of the action of $N$ on $G/G'$ follows directly from the relation between the canonical map and the action of $\tilde L\ot H$ on $\Hom_K(L,\tilde L)$.

\eqref{GP:e}
This follows directly from \eqref{eq:compGH}, with $hG' \in G/G'$ corresponding to $f \in \Hom_K(L,\tilde{L})$ and $n$ corresponding to an element of $\tilde L \ot H$.

\eqref{GP:g}
Denote $\alpha:N\to \Perm(G/G')$ and $\lambda:G\to \Perm(G/G')$, then by \eqref{GP:e}, we find for all $g,h\in G$ and $n\in N$
$$\lambda(g)\alpha(n)\lambda(g^{-1})(hG')= g\cdot ((g^{-1}hG')\cdot n) = (hG')\cdot (g\cdot n) = \alpha(g\cdot n)(hG'),$$
and therefore $N$ is normalized by $G$ in $\Perm(G/G')$.
\end{proof}

In \cite{GreitherPareigis87}, this theorem is used to show that $H$-Galois structures on the extension $L/K$ are in bijective correspondence with subgroups $N$ of $\Perm(G/G')$ that act strictly transitive on $G/G'$ and that are normalized by $\lambda(G)$, which is the permutation of $G/G'$ given by left multiplication.
We will make use of this result in order to classify intermediate Hopf-Galois extensions.

\subsection{Intermediate extensions}\label{intermediate}
The aim of this section is to establish, for a finite separable $H$-Galois extension $L/K$, the following proposition characterising the fixed field of left ideals two-sided coideals as $H$-subextensions, which will allow us to prove the main Theorem~\ref{theo:fin-cor} in the finite case.

\begin{prop}\label{prop:main}
Let $L/K$ be a finite separable $H$-Galois extension, where $H$ is a Hopf $K$-algebra.
Then for all left ideals two-sided coideals $I \subseteq H$, the following statements hold:
\begin{enumerate}[(a)]
\item \label{prop:main:a} The fixed field $L^I$ is an $H$-subextension.
\item \label{prop:main:b} If $I$ is a Hopf ideal, then the extension $L^I/K$ is $H/I$-Galois.
\end{enumerate}
\end{prop}

\begin{lem}\label{H/I=N/V}
There is a bijective correspondence between the sets of:
\begin{itemize}
\item left ideals two-sided coideals $I \subseteq H$, and
\item $G$-equivariant subgroups $V \subseteq N$
\end{itemize}
Under this correspondence, the isomorphism $\tau:\tilde L\ot_K H\to \tilde L[N]$ from \ref{th:GP}\eqref{GP:a} induces an isomorphism of left $\tilde{L} \ot H$-module $\tilde{L}$-coalgebras
\[\tau: \tilde{L} \otimes_K H/(\tilde{L} \ot I) \cong \tilde{L} \otimes_K H/I \to \tilde{L}[N/V].\]
Moreover, $V$ is a normal subgroup of~$N$ if and only if $I \subset H$ is a Hopf ideal. In that case, $\tau$ is an isomorphism in the category of $\tilde{L}$-Hopf algebras.
\end{lem}

\begin{proof}
Applying extension of scalars on a left ideal two-sided coideal $I\subset H$, we obtain a left ideal two-sided coideal $\tilde L\ot I$ in $\tilde L\ot H\cong \tilde L[N]$, so the existence of a subgroup $V \subseteq N$ follows from Proposition~\ref{grouplikesHopf}.
Recall from Theorem \ref{th:GP}\eqref{GP:b} that the $G$-action on $N$ is induced by the action of $G$ on $\tilde L\ot H$ on the first tensorant. Since the quotient $\tilde{L} \otimes H \twoheadrightarrow \tilde{L} \otimes H/I$ is clearly $G$-equivariant under this action, $V$ becomes a $G$-equivariant subgroup of~$N$.

Conversely, if $V\subseteq N$ is a $G$-equivariant subgroup, then $\tilde L[V]$ is a Hopf subalgebra of $\tilde L[N]$ in the symmetric monoidal category of $\tilde L\# K[G]$-modules. Applying Proposition~\ref{pr:descent}, we find that the space of $G$-invariants $H'=\tilde L[V]^G$ is a Hopf subalgebra of $H$ in $\Vect_K$. The desired left ideal two-sided coideal is then given by $I=H H'^+$ (see Theorem~\ref{theo:subalgebra-ideal}). The correspondence is indeed bijective, as it is based on the equivalence of categories from Proposition~\ref{pr:descent}.

The last statement follows directly from the well-known fact that those left ideals two-sided coideals in a group algebra that are Hopf ideals, correspond to those subgroups that are normal.
\end{proof}

We now include a purely group theoretic lemma, which we state here in its profinite version needed later.

\begin{lem}\label{lem:gp-act}
Let $G$ be a profinite group and $G' \subseteq G$ a closed subgroup.
Further, let $N$ be a (possibly non-commutative) profinite left $G$-module which acts continuously on $G/G'$ on the right such that for all $n \in N$, all $g,h \in G$ we have
\[ g \cdot \big( (hG') \cdot n \big) = (ghG') \cdot (g\cdot n) \]
and, in particular, for $h=1$,
\begin{equation}\label{eq:gn}
g \cdot \big( (1G') \cdot n \big) = (gG') \cdot (g \cdot n)
\end{equation}
We assume that the orbit map
\[ \beta: N \to G/G', \;\;\; n \mapsto (1G').n \]
is a homeomorphism. It is trivially right $N$-equivariant.
Finally, let $V \subseteq N$ be a $G$-equivariant open (resp.\ closed) subgroup $V \subseteq N$ and define
\[ U = \{ u \in G \;|\; uG' \in \beta(V)\}.\]

Then the following statements hold.
\begin{enumerate}[(a)]
\item \label{lem:gp-act:a} $U$ is an open (resp.\ closed) subgroup of~$G$, which contains $G'$. Moreover, if $V \subseteq N$ is open, then $G/U$ is finite.

\item \label{lem:gp-act:b} For all $g \in G$ and all $v \in V$ we have
\[ \pr_U\big(  (gG') \cdot v \big) = \pr_U ( gG' ) = gU, \]
where $\pr_U: G/G' \twoheadrightarrow G/U$ denotes the natural projection.

\item \label{lem:gp-act:c} The composition $N \xrightarrow{\beta} G/G' \overset{\pr_U}{\twoheadrightarrow} G/U$ induces a homeomorphism $N/V \to G/U$.

\end{enumerate}
\end{lem}

\begin{proof}
\eqref{lem:gp-act:a} If $\pi: G \to G/G'$ denotes the natural projection, then $U = \pi^{-1}\big(\beta(V)\big)$, showing both the the containment $G' \subseteq U$ and the open-ness (closed-ness) of~$U$.
Let now $u_1,u_2 \in U$ and $v_1,v_2 \in V$ such that $u_1 G = (1G') \cdot v_1$ and $u_2 G = (1G') \cdot v_2$.
Then by \eqref{eq:gn} we have
\[ (1G')\cdot \big(v_1 (u_1 \cdot v_2)\big) = \big((1G') \cdot v_1\big) (u_1 \cdot v_2) = (u_1 G') \cdot (u_1 \cdot v_2) = u_1 \big( (1G') \cdot v_2 \big) = u_1u_2 G',\]
showing $u_1,u_2 \in U$ as $v_1 (u_1 \cdot v_2) \in V$.
Applying \eqref{eq:gn} with $g= u^{-1}$ and $n=v$ for $uG' = (1G') \cdot v$ we obtain
\[ 1G' = u^{-1} \big( (1G') \cdot v \big) = (u^{-1}G') \cdot (u^{-1} \cdot v),  \]
which shows $u^{-1} \in U$ by acting on both sides with $(u^{-1} \cdot v)^{-1}$.
It is well known that open subgroups are of finite index.

\eqref{lem:gp-act:b}
We take $u \in U$ such that $(1G') \cdot (g^{-1} \cdot v) = uG'$ and apply again \eqref{eq:gn} to get
\[ (gG') \cdot v  = (gG') \cdot \big(g \cdot (g^{-1} \cdot v)\big) = g \big( (1G') \cdot (g^{-1} \cdot v) \big) = guG'. \]

\eqref{lem:gp-act:c}
The composite map $\pr_U \circ \beta$ is surjective, continuous and open. First let $n \in N$ and $v \in V$. Then there are $g \in G$ and $u \in U$ with $(1G') \cdot n=gG'$ and $(1G') \cdot v=uG'$. We then have
\begin{multline*}
 \pr_U \circ \beta(nv) = \pr_U\big( (1G') \cdot (nv) \big) = \pr_U \big( ((1G') \cdot n) \cdot v \big) = \pr_U \big( (gG') \cdot v \big)\\
 = \pr_U \big( gG' \big) = \pr_U \big( (1G') \cdot n \big) = \pr_U \circ \beta(n),
\end{multline*}
showing that the map indeed descends to a surjective map $N/V \twoheadrightarrow G/U$.

Assume now $\pr_U \circ \beta (n) = \pr_U \circ \beta (m)$ for $m,n \in N$. We then have $g \in G$ and $u \in U$ such that $(1G') \cdot m=gG'$ and $(1G') \cdot n = guG'$. Let $v \in V$ be such that $(1G') \cdot v=uG'$.
But we have
\[ guG' = g(uG') = g\big((1G') \cdot v\big) = (gG') \cdot (g \cdot v) = ((1G') \cdot m) \cdot (g \cdot v) = (1G') \cdot (m \cdot (g \cdot v)). \]
Because of the injectivity of $\beta$, it thus follows $n = m \cdot (g \cdot v)$, showing $nV = mV$, as required.
\end{proof}

\begin{lem}\label{lem:main}
Let $I \subseteq H$ be a left ideal two-sided coideal. Consider the associated $G$-equivariant subgroup $V\subset N$ (where as before $\tilde L\ot H\cong \tilde L[N]$) as in Lemma \ref{H/I=N/V} and $U\subset G=\Gal(\tilde L/K)$ the associated subgroup containing $G'=\Gal(L/K)$ as in Lemma~\ref{lem:gp-act}.

We have $L^I = L^U$ and the natural map
\[ \tilde L \otimes_K H/I \to \Hom_K(L^I,\tilde{L})\]
is an isomorphism.
\end{lem}

\begin{proof}
Theorem~\ref{th:GP} ensures that the hypotheses of Lemma~\ref{lem:gp-act} are fulfilled.
From Lemmas \ref{lem:gp-act} and \ref{H/I=N/V}, it follows that we have the $K$-linear isomorphism\footnote{More precisely, it is an isomorphism of $\tilde L$-coalgebras and right $\tilde L\ot H$-linear maps.}
\[ \tilde L \otimes_K H/I \cong \tilde L [N/V] \cong  \tilde L[G/U] \cong \Hom_K(L^U,\tilde{L}). \]
Let now $i \in I$ and $f \in \Hom_K(L,\tilde L)$ be injective (for instance, a field homomorphism). Under the canonical map, $f$ corresponds to an element $\psi$ in $\tilde L \ot H$. It is clear that $\psi \cdot (1 \otimes i)$ is zero in $\tilde L \otimes H/I$. This, however, means that the restriction to $L^U$ of $f \cdot (1\otimes i)$ is the zero map. In particular, we see for all $x \in L^U$
\[ \big(f \cdot (1 \ot i)\big)(x) = f(i \cdot x) = 0.\]
The injectivity of $f$ now implies $i \cdot x=0$ for all $x \in L^U$. We conclude $L^U \subseteq L^I$.
The proof is now concluded by observing that the composite
\[ \tilde L \otimes H/I \twoheadrightarrow \Hom_K(L^I,\tilde L) \overset{\res}{\twoheadrightarrow} \Hom_K(L^U,\tilde L) \]
is exactly the isomorphism above, whence the last epimorphism is bijective and therefore $L^I = L^U$.
\end{proof}

\begin{proof}[Proof of Proposition~\ref{prop:main}.]
From the isomorphism of Lemma~\ref{lem:main}, Lemma \ref{lem:extend-rationality} directly gives the statement~\eqref{prop:main:a}.
The statement~\eqref{prop:main:b} is by definition.
\end{proof}

We can now prove our main theorem for finite separable Hopf-Galois extensions.

\begin{theo}\label{theo:fin-cor}
Let $L/K$ be a separable finite $H$-Galois extension.
Then the following maps are inclusion reversing bijections:
\[
\xymatrix@R0.5cm{
\big\{L/L_0/K \st L_0 \t{ $H$-subextension }\big\} \ar@<.5ex>[r]^-\Phi & \big\{I \subset H \st I \t{ left ideal, two-sided coideal }\big\} \ar@<.5ex>[l]^-{\Psi}\\
L_0 \ar@{|->}[r] & J(L_0)\\
L^I  & I \ar@{|->}[l]\\
}\]
Moreover, the above correspondence restricts to a bijection between the following subsets
\[
\xymatrix@R0.5cm{
\big\{L/L_0/K \st L_0 \t{ $H$-normal }\big\} \ar@<.5ex>[r]^-\Phi
& \big\{I \subset H \st I \t{ Hopf ideal }\big\}
\ar@<.5ex>[l]^-{\Psi} \\
}\]
\end{theo}

\begin{proof}[Proof of Theorem~\ref{theo:fin-cor}.]
We already know that the map $\Phi$ is well-defined by Proposition~\ref{prop:HopfIdeal}. On the other hand, for any left ideal two-sided coideal $I$ of $H$, we define $\Psi(I)=L^I$, and we know now by Proposition~\ref{prop:main} that $L^I$ is an $H$-subextension. If, in addition, $I$ is a right ideal, then $L^I$ is also $H$-stable (by Lemma~\ref{prop:intermediate}) and, thus, $H$-normal. This shows that $\Psi$ is well-defined as well.

Furthermore, Proposition~\ref{prop:main} and Lemma~\ref{lem:extend-rationality} tell that $\Phi\circ \Psi(I) = J(L^I)=I$, which provides the first half of the correspondence.

For the other half of the correspondence, let $L_0$ be an $H$-subextension. We clearly have the inclusion $L_0 \subseteq L^{J(L_0)}$.
Moreover, by the definition of $H$-subextensions, the natural map $L \otimes H/J(L_0) \to \Hom_K(L_0,L)$ is an isomorphism,
and on the other hand by Lemma~\ref{lem:main}, $\tilde L \otimes H/J(L_0) \to \Hom_K(L^{J(L_0)},\tilde L)$ is also an isomorphism. We hence conclude that $L_0=L^{J(L_0)}$.
\end{proof}

Finally, let us rephrase our correspondence in purely group-theoretical terms via Greither-Pareigis theory.

\begin{cor}\label{co:correspGP}
Combining Theorem~\ref{theo:fin-cor} with Lemma~\ref{H/I=N/V}, we obtain also a bijective correspondence as follows
\[
\xymatrix@R0.5cm{
\big\{L/L_0/K \st L_0 \t{ $H$-subextension }\big\} \ar@<.5ex>[r]^-{\Phi'} & \big\{V \subset N \st V \t{ is a $G$-equivariant subgroup}\big\} \ar@<.5ex>[l]^-{\Psi'} 
}\]
restricting to another bijective correspondence
\[
\xymatrix@R0.5cm{
\big\{L/L_0/K \st L_0 \t{ $H$-normal }\big\} \ar@<.5ex>[r]^-{\Phi'} 
& \big\{V \subset N \st V \t{ is a $G$-equivariant normal subgroup}\big\} 
\ar@<.5ex>[l]^-{\Psi'} \\
}\]
As in Theorem~\ref{th:GP}\eqref{GP:g} the $G$-equivariance of the subgroup can also be rephrased as being normalized by $\lambda(G)$ in $\Perm(G/G')$.
\end{cor}

\subsection{Hopf-Galois theory for subextensions}\label{HGtheory}

In classical Galois theory, any intermediate field $L_0$ of a Galois extension $L/K$ leads to a Galois extension $L/L_0$. Our next result shows that the same property holds for an $H$-subextension of an $H$-Galois extension.

\begin{proposition}\label{prop:subs}
Let $L/K$ be a finite $H$-Galois extension and let $L_0$ be an $H$-subextension. Consider the Hopf subalgebra $H_0=\Psi'(L_0)$ as constructed in Corollary~\ref{cor:Phi'}. Denote as before $\alpha:H\to \End_K(L)$ the map associated to the action of $H$ on $L$. Then the following statements hold.
\begin{enumerate}[(a)]
\item \label{prop:subs:a} $L \otimes_K H_0 = \{ \sum_i x_i \otimes h_i \in L \otimes_K H \;|\; \can(\sum_i x_i \otimes h_i) \in \End_{L_0}(L)\}$.
\item \label{prop:subs:b} $L/L_0$ is $L_0\ot H_0$-Galois, i.e.\ the map $\can_{L/L_0}:L\ot H_0\to \End_{L_0}(L)$ is bijective.
\item \label{prop:subs:c} If $L_0/K$ is moreover $H$-normal, then $L_0/K$ is also $H/J(L_0)$-Galois.
\end{enumerate}
\end{proposition}

\begin{proof}
\eqref{prop:subs:a}
If $x\ot h\in L\ot H_0$, then we find by Corollary~\ref{cor:Phi'} that $\can(x\ot h)=x\alpha(h)\in \End_{L_0}(L)$.
Conversely, take $\sum_i x_i \otimes h_i \in L \otimes_K H$ such that the elements $x_i$ are linearly independent over $K$ and  
$\can(\sum_i x_i \otimes h_i)$ is left $L_0$-linear. This means that for all $x_0\in L_0$ and all $x\in L$, the following equality holds
$$\sum_i x_ix_0 (h_i\cdot x)=\sum_i x_ih_i\cdot (x_0x) = \sum_i x_i(h_{i(1)}\cdot x_0)(h_{i(2)}\cdot x)$$
Using the bijectivity of $\can:L\ot H\to \End_K(L)$, this equality can be translated into
$$\sum_i x_ix_0 \ot h_i =  \sum_i x_i(h_{i(1)}\cdot x_0) \ot h_{i(2)}$$
which holds for all $x_0\in L_0$. Since $L_0$ is an $H$-subextension, the map $\can_0:L\ot H/J(L_0)\to \Hom_K(L_0,L)$ is injective. Hence the previous equality is furthermore equivalent to
$$\sum_i x_i\ot \pi(1) \ot h_i =  \sum_i x_i \ot \pi(h_{i(1)}) \ot h_{i(2)}\in L\ot H/J(L_0)\ot H$$
where we denote $\pi:H\to H/J(L_0)$ for the canonical surjection. Since we assumed that the elements $x_i$ are linearly independent, we find that $\pi(1) \ot h_i = \pi(h_{i(1)}) \ot h_{i(2)}$. Hence we find by Theorem~\ref{theo:subalgebra-ideal} that $h_i\in H_0$.\\
\eqref{prop:subs:b}
Consider the following commutative diagram
\[
\xymatrix{
L\ot H \ar[rr]^-\can_\sim && \End_K(L)\\
L\ot H_0 \ar@{^(->}[u] \ar[r]^-\cong &  L\ot_{L_0} (L_0\ot H_0) \ar[r]^-{\can_{L/L_0}} & \End_{L_0}(L) \ar@{^(->}[u]
}
\]
Since $\can$ is bijective, $\can_{L/L_0}$ is also injective. It is surjective by part~\eqref{prop:subs:a}. Hence $L/L_0$ is $L_0\ot H$-Galois.\\
\eqref{prop:subs:c} This was already proved as Proposition~\ref{prop:HopfIdeal}\eqref{prop:HopfIdeal:c}.
\end{proof}

Next we treat the passage to intersections and composita. We see that it is very convenient to have the possibility to work with either ideals or subalgebras.

\begin{lem}\label{lem:interm}
Let $L/K$ be a finite field extension and let $K \subseteq L_i\subseteq L$ for $i=1,2$ be intermediate fields.
\begin{enumerate}[(a)]
\item \label{lem:interm:a} $\End_{L_1L_2}(L)=\End_{L_1}(L)\cap \End_{L_2}(L)$.
\item \label{lem:interm:b} If $\End_{L_1}(L) = \End_{L_2}(L)$, then $L_1=L_2$.
\end{enumerate}
\end{lem}

\begin{proof}
\eqref{lem:interm:a}
Clearly, an endomorphism of $L$ is both $L_1$- and $L_2$-linear if and only if it is $L_1L_2$-linear. \\
\eqref{lem:interm:b}
By part~\eqref{lem:interm:a} we find that $\End_{L_1}(L) = \End_{L_2}(L)=\End_{L_1L_2}(L)$. Furthermore, we have $\dim_K(\End_{L'}(L))=[L:K][L:L']$ for any intermediate field $K\subset L'\subset L$. Hence it follows that $L_1$, $L_2$ and $L_1L_2$ have the same degree over $K$. Since obviously $L_i\subset L_1L_2$ for $i=1,2$, it follows that $L_1=L_2=L_1L_2$.
\end{proof}

\begin{prop}\label{prop:intercomp}
Let $L/K$ be a separable finite $H$-Galois extension. Let $L_1 = L^{I_1} = L^{H_1}$ and $L_2 = L^{I_2} = L^{H_2}$ be $H$-subextensions corresponding to left ideals two-sided coideals $I_1$, $I_2$ and Hopf-subalgebras $H_1$, $H_2$, respectively.
\begin{enumerate}[(a)]
\item \label{prop:intercomp:a} $L_1 L_2$ is an $H$-subextension and $L_1 L_2 = L^{H_1 \cap H_2}$.
If $L_1$ and $L_2$ are $H$-normal, then so is $L_1 L_2$.
\item \label{prop:intercomp:b} $L_1 \cap L_2$ is an $H$-subextension and $L_1 \cap L_2 = L^{I_1 + I_2}$.
If $L_1$ and $L_2$ are $H$-normal, then so is $L_1 \cap L_2$.
\end{enumerate}
\end{prop}

\begin{proof}
First remark that by Lemma~\ref{lem:sub-basics} $I_1+I_2$ and $H_1 \cap H_2$ are of the right type.\\
\eqref{prop:intercomp:a}
By Proposition~\ref{prop:subs}, we have isomorphisms
$$\can_{L/L_i}: L \otimes_K H_i\to \End_{L_i}(L)$$
for $i=1,2$.
Taking intersections on both sides and using Lemma~\ref{lem:interm} we obtain a canonical isomorphism
$$L \otimes_K (H_1 \cap H_2)\to \End_{L_1L_2}(L).$$
As $H_1 \cap H_2$ is a Hopf-subalgebra, it corresponds  to a unique $H$-subextension $L_3=L^{H_1\cap H_2}$ with canonical isomorphism $L \otimes_K (H_1 \cap H_2) \cong \End_{L_3}(L)$. Lemma~\ref{lem:interm}\eqref{lem:interm:b} implies $L_3 = L_1 L_2$, concluding the statement about composita.
Since we deal with $H$-subextensions, the $H$-normality reduces to $H$-stability, which is clear for $L_1L_2$.\\
\eqref{prop:intercomp:b}
Consider the following commutative diagram.
\[
\xymatrix{
&& L\ot H \ar[d]^{\can}_\sim \ar@{->>}[ddll] \ar@{->>}[ddrr] \\
&& \End_K(L) \ar@{->>}[dl] \ar@{->>}[dr] \\
L\ot H/I_1 \ar[r]^-{\can_1}_-\sim \ar@{->>}[ddrr] &\Hom_K(L_1,L) \ar@{->>}[dr] && \Hom_K(L_2,L) \ar@{->>}[dl] 
& L\ot H/I_2 \ar[l]_-{\can_2}^-\sim \ar@{->>}[ddll]
\\
&& \Hom_K(L_1\cap L_2,L)\\
&& L\ot H/(I_1+I_2) \ar[u]_{\can_{1,2}}
}
\]
The inner part of this diagram is a pushout square (in $\Vect_K$), since it arises by applying the contravariant functor $\Hom_K(-,L)$ to the diagram which expresses the intersection $L_1\cap L_2$ as the pullback of the inclusions $L_i\to L$.
The outer part of the diagram is also a pushout square. 
Since $\can$, $\can_1$ and $\can_2$ are isomorphisms, it follows that $\can_{1,2}$ is also an isomorphism.
Moreover, as $L_1\cap L_2\subset L^{I_1+I_2}$, we have that $\can_{1,2}$ factors through the canonical map $L\ot H/(I_1+I_2)\to \Hom_K(L^{I_1+I_2},L)$, which is an isomorphism since $L^{I_1+I_2}$ is an $H$-subextension by Proposition~\ref{prop:main}. It follows that $\Hom_K(L^{I_1+I_2},L) \cong \Hom_K(L_1\cap L_2,L)$ and therefore $L_1\cap L_2=L^{I_1+I_2}$ is an $H$-subextension.
\end{proof}

\section{Infinite Hopf-Galois Theory}

\subsection{Proartinian Hopf algebras}\label{sec:proartinian}

In this section, we present the underlying algebraic language in which we formulate our infinite Hopf-Galois correspondence.
Let $K$ be any field, equipped with the discrete topology. Let $\Vect_K$ be the category of finite dimensional $K$-vector spaces.
We now extend it by introducing proartinian $K$-vector spaces.
First we recall that a topological $K$-vector space $V$ is a $K$-vector space equipped with a topology with respect to which both addition and scalar multiplication are continuous.

\begin{defi}\label{defi:proartinian}
A {\em proartinian $K$-vector space} is a topological $K$-vector space such that there exists a family $\Lambda_V$ of open subspaces $U\subseteq V$ which form a base of open neighbourhoods of $0$ such that the following two statements hold:
\begin{enumerate}[(1)]
\item\label{defi:proartinian:1} the quotient $V/U$ is a finite dimensional discrete $K$-vector space for each $U \in \Lambda_V$,
\item\label{defi:proartinian:2} the natural $K$-linear map $\pi^V: V \to\plim{U \in \Lambda_V} V/U$, induced by all projection maps $\pi^V_U=\pi_U:V\to V/U$, is bijective.
\end{enumerate}
A morphism between proartinian $K$-vector spaces is a continuous $K$-linear map.
The category of proartinian $K$-vector spaces is denoted $\ProVect_K$.
\end{defi}

The category $\ProVect_K$ enjoys the following further properties (which can also be deduced, for instance, from \cite[I,\S3]{Fontaine}, \cite{Takeuchi85}).

\begin{lem}\label{lem:proartinian}
Let $V$ be a proartinian $K$-vector space.
\begin{enumerate}[(a)]
\item\label{lem:proartinian:a}
For every open subspace $W \subseteq V$, the quotient $V/W$ is a finite dimensional vector space with the discrete topology, the projection $\pi_W: V \to V/W$ is continuous, and $W$ is closed.
\item\label{lem:proartinian:b}
If one endows the projective limit $\plim{U \in \Lambda_V} V/U$ with the coarsest topology for which all projection maps are continuous, the linear isomorphism $\pi^V$ of condition \eqref{defi:proartinian:2} becomes a homeomorphism. We will henceforth often identify $V$ with this projective limit along this isomorphism.
\item\label{lem:proartinian:c}
The set $\bigcap_{U \in \Lambda_V} U = \{0\}$ is closed and the morphisms $\pi_{U}$ are jointly monic.
\end{enumerate}
\end{lem}

\begin{proof}
\eqref{lem:proartinian:a}
As $W$ is open, we have $U \subseteq W$ for some $U \in \Lambda_V$ and the finite dimensionality of $V/W$ follows from the one of $V/U$ via the natural surjection $V/U \twoheadrightarrow V/W$. The discreteness in the quotient topology is clear. The other two statements follow from $\pi_W^{-1}(\{0\}) = W$, once viewing $\{0\}$ as an open set and once as a closed one.

\eqref{lem:proartinian:b}
By \eqref{lem:proartinian:a}, $\pi^V$ identifies $\Lambda_V$ with a base of open neighbourhoods of~$0$ of $\plim{U \in \Lambda_V} V/U$.

\eqref{lem:proartinian:c} The equality $\bigcap_{U \in \Lambda_V} U = \{0\}$ follows from the injectivity of $\pi^V$ and also means that the $\pi_U$ are jointly monic.
Further, as any open subspace $U \subseteq V$ is also closed by \eqref{lem:proartinian:a}, $\bigcap_{U \in \Lambda_V} U$ is indeed closed.
\end{proof}

\begin{prop}\label{prop:ProVect}
Let $V$ be a proartinian $K$-vector space and $W \subseteq V$ be a subspace.
\begin{enumerate}[(a)]
\item \label{prop:ProVect:a}
The canonical inclusion of $W$ in $V$ factors as in the following diagram of injective $K$-linear maps
\[ \xymatrix{ W \ar@{^(->}[d]^{\pi^W} \ar@{^(->}[rr] && V \ar[d]_\cong^{\pi^V} \\
\plim{U \in \Lambda_V} W/(U \cap W) \ar@{^(->}[rr] && \plim{U \in \Lambda_V} V/U }\]
\item \label{prop:ProVect:a2}
The topological closure of $W$, denoted $\widehat{W}$, is given as
\[ \widehat{W} = \bigcap_{U \in \Lambda_V} (U+W)  \cong \plim{U \in \Lambda_V} (U+W)/U \cong \plim{U \in \Lambda_V} W/(U \cap W).\]
Consequently, $\pi^W$ is a bijection if and only if $W$ is closed.
\item \label{prop:ProVect:b}
$W \subseteq V$ is a closed subspace if and only if the quotient $V/W$ is a proartinian $K$-vector space and the natural projection $\pi_W:V\to V/W$ is continuous. In this case the projection map is also an open $K$-linear map.
\end{enumerate}
\end{prop}

\begin{proof}
\eqref{prop:ProVect:a}
Since $W/(U \cap W)$ is a subspace of $V/U$ for every $U \in \Lambda_V$, the natural map $\plim{U \in \Lambda_V} W/(U \cap W) \to \plim{U \in \Lambda_V} V/U$ is injective. The injectivity of $\pi^W$ follows from Lemma~\ref{lem:proartinian}\eqref{lem:proartinian:c}.

\eqref{prop:ProVect:a2}
First observe that $\bigcap_{U \in \Lambda_V} (U+W)$ is closed as it is the intersection of open (and hence closed) subspaces $U+W$.
Let $W'$ be a closed subspace containing~$W$. We have to show that $W'$ contains $\bigcap_{U \in \Lambda_V} (U+W)$.
Ad absurdum, let $v \in  \bigcap_{U \in \Lambda_V} (U+W) \setminus W'$. Since $W'$ is closed, there is $U \in \Lambda_V$ such that $(v+U) \cap W' = \emptyset$.
However, our assumption implies that $v\in (U+W)\setminus W'$, so $v=u+w$ for some $u\in U$ and $w\in W$. Therefore $W'\supset W \ni w=v-u \in v+U$, which provides the required contradiction.
This shows $\widehat W=\bigcap_{U \in \Lambda_V} (U+W)$.

For $U,U' \in \Lambda_V$ with $U \subseteq U'$, we have the commutative diagram
\[
\xymatrix{
U+W  \ar@{->>}[r] \ar@{^(->}[d] & (U+W)/U \ar@{->>}[d] \\
U'+W \ar@{->>}[r]        & (U'+W)/U'.}
\]
The projective limit along $U \in \Lambda_V$ hence gives $\bigcap_{U \in \Lambda_V} (U+W) \twoheadrightarrow \plim{U \in \Lambda_V} (U+W)/U$, the kernel of which is immediately checked to be $\bigcap_{U \in \Lambda_V} U = \{0\}$, which shows the first stated isomorphism.

The last isomorphism is a consequence of the classical one $(U+W)/U \cong W/(U\cap W)$. The final statement follows then from the fact that composition of the inclusion $W\subset \widehat W$ with the stated isomorphisms is exactly the injective morphism $\pi^W$ from part~\eqref{prop:ProVect:a}.

\eqref{prop:ProVect:b}
If $W$ is closed, by part \eqref{prop:ProVect:a2}, we find that
\[ V/W = V/( \bigcap_{U \in \Lambda_V} (U+W) ) \cong \plim{U \in \Lambda_V} V/(U+W) \cong  \plim{U \in \Lambda_V} (V/W) / ((U+W)/W), \]
which shows that $V/W$ is indeed proartinian with base of open neighbourhoods of~$0$ given by $\pi_W(U) = (U+W)/W$ for $U \in \Lambda_V$, and that $\pi_W$ is an open map.

Conversely, if $V/W$ is proartinian, $\{0\}$ is closed, so $V/W\setminus\{0\}$ is open in $V/W$.
Since $\pi_W$ is continuous, $V\setminus W = \pi_W^{-1}(V/W\setminus\{0\})$ is open in $V$ and hence $W$ is closed in $V$.
\end{proof}

Proartinian $K$-vector spaces enjoy the following further properties.
\begin{prop}\label{prop:proartinian}
\begin{enumerate}[(a)]
\item\label{prop:proartinian:a} Any proartinian $K$-vector space is Hausdorff.
\item\label{prop:proartinian:b} The discrete proartinian $K$-vector spaces are exactly the finite dimensional $K$-vector spaces.
\item\label{prop:proartinian:c} A closed subspace $W \subseteq V$ is open if and only if it is of finite codimension.
\end{enumerate}
\end{prop}

\begin{proof}
Let $V$ be a proartinian $K$-vector space.

\eqref{prop:proartinian:a}
Any two vectors $v,w \in V$, $v \neq w$ have the disjoint open neighbourhoods $v+U$ and $w+U$, respectively, for any open subspace $U \in \Lambda_V$ not containing $v-w$.

\eqref{prop:proartinian:b} Suppose $V$ is finite dimensional. Then the intersection $\bigcap_{U \in \Lambda_V} U = \{0\}$ of Lemma~\ref{lem:proartinian}\eqref{lem:proartinian:c} can be refined to a finite intersection of open sets, showing that $\{0\}$ is open, whence $V$ is discrete.
Conversely, if $V$ is discrete, then $\{0\}$ is an open subspace and $V \cong V/\{0\}$ is finite dimensional by Lemma~\ref{lem:proartinian}\eqref{lem:proartinian:a}.

\eqref{prop:proartinian:c} If $W \subseteq V$ is open, then $V/W$ is finite dimensional by Lemma~\ref{lem:proartinian}\eqref{lem:proartinian:a}.
Conversely, if $V/W$ is finite dimensional, then it is a discrete proartinian $K$-vector space and consequently $W = \pi_W^{-1}(\{0\})$ is open.
\end{proof}

For a $K$-linear map $V\to W$ between proartinian vector spaces, the continuity can be expressed as the following property. For any $U\in \Lambda_W$, there exists a $U'\in \Lambda_V$ such that $U' \subseteq \ker\pi_{U}\circ f$, which means that there exists a $K$-linear map $f_{U,U'}:V/{U'} \to W/U$ such that the following diagram commutes
\[
\xymatrix{
V \ar[rr]^-f \ar@{->>}[d]_{\pi_{U'}} && W \ar@{->>}[d]^-{\pi_{U}} \\
V/{U'} \ar[rr]^-{f_{U,U'}} && W/{U}.
}
\]
We note $f_{U,U'} \circ \pi_{U', U''} = f_{U,U''}$ for any $U'' \in \Lambda_V$ such that $U'' \subseteq U'$ and the natural projection $\pi_{U', U''}: V/U'' \twoheadrightarrow V/U'$.
Since the morphisms $\pi_{U}$ are jointly monic, the family of morphisms $f_{U,U'}$ completely determines the map $f$. Nevertheless, some caution is needed when one wants to express properties of the morphism $f$ in terms of the morphisms $f_{U,U'}$ defined on finite quotients. Indeed, the continuity of $f$ requires the existence of such morphisms $f_{U,U'}$, however the choice of these morphisms is not unique (nor canonical). Depending on the choices made, the morphisms on finite quotients do not always properly reflect the properties of the morphism $f$, as the following example shows.

\begin{ex}\label{ex:non-iso}
Denote by $K^\infty = \prod_{n \ge 1} K = \plim{n \ge 1} K^n$ taken with respect to the maps $\pi_{n,m} : K^m \to K^n$ for $m \ge n$, projecting onto the first $n$ components. Further let $\pi_n:=\pi_{n,\infty}:K^\infty\to K^n$ be the natural projection map.

Consider the identity map on $K^{\infty}$. To express the continuity of the identity map the most intuitive choice of morphisms on finite quotients would be to take $id_{K^n}$ and the following commutative diagrams:
$$\xymatrix{
K^\infty \ar[rr]^{id} \ar@{->>}[d]^{\pi_{n}} && K^\infty \ar@{->>}[d]^-{\pi_n} \\
K^{n} \ar[rr]^{id_{K^n}} && K^n
}$$
However, this is not the only possible choice. Indeed, one could also consider for any projection $\pi_n$ on the codomain, a projection $\pi_{n+1}$ on the domain together with the projection $\pi_{n,n+1}$ as morphism on finite quotients. This also expresses the continuity of the identity map by the commutativity of the following diagram:
$$\xymatrix{
K^\infty \ar[rr]^{id} \ar@{->>}[d]^{\pi_{n+1}} && K^\infty \ar@{->>}[d]^-{\pi_n} \\
K^{n+1} \ar[rr]^{\pi_{n,n+1}} && K^n
}$$
Nevertheless, for this choice, the morphisms $\pi_{n+1,n}$ are surjective but not injective, showing that depending of the choice, morphisms on finite quotients of an isomorphism are not necessarily isomorphisms themselves.
\end{ex}

It will be important for us that $\ProVect_K$ is equipped with the completed tensor product $\hot$ over~$K$, making it into a symmetric monoidal category. More precisely, given two proartinian $K$-vector spaces $V$ and $W$, one defines
\[V\hot W=\plim{(U_V,U_W)\in \Lambda_V\times \Lambda_W} V/U_V\otimes W/U_W.\]
For any pair $(U_V,U_W)\in \Lambda_V\times \Lambda_W$, we denote the corresponding projection map $V\hot W\to V/U_V\ot W/U_W$ by $\pi_{U_V}\hot \pi_{U_W}$.
Moreover, there is a canonical $K$-linear monomorphism $V\ot W \to V\hot W$.

We now put Hopf algebras and coalgebras into the context of proartinian vector spaces.

\begin{defi}\label{defi:infHopf}
Let $H \in \ProVect_K$.
We say that $H$ is a {\em proartinian Hopf algebra over~$K$} if there is a base $\Lambda_H$ of $0$ consisting of open subspaces $U$ of~$V$ such that for every $U \in \Lambda_H$, the quotient $V/U$ is a finite dimensional $K$-vector space equipped with the structure of a Hopf algebra over~$K$ and for all $U,V \in \Lambda_H$ with $U \subseteq V$ the natural maps $\pi_{V,U}: H/U \to H/V$ are morphisms of Hopf algebras over~$K$ with respect to the fixed Hopf algebra structures.

Similarly, we say that $H$ is a {\em proartinian coalgebra over~$K$} if there is a base $\Lambda_H$ of $0$ consisting of open subspaces $U$ of~$V$ such that for every $U \in \Lambda_H$, the quotient $V/U$ is a finite dimensional $K$-vector space equipped with the structure of a coalgebra over~$K$ and for all $U,V \in \Lambda_H$ with $U \subseteq V$ the natural maps $\pi_{V,U}: H/U \to H/V$ are morphisms of coalgebras over~$K$ with respect to the fixed coalgebra structures.

A morphism $f: H \to H'$ between proartinian Hopf algebras (resp.\ coalgebras) is a $K$-linear map such that for each $U' \in \Lambda_{H'}$, there exists $U \in \Lambda_H$  with $f(U)\subset U'$ such that the induced map $f_{U',U}:H/U\to H'/U'$, which satisfies $\pi_{U'}\circ f=f_{U',U} \circ \pi_{U}$, is a morphism of Hopf algebras (resp.\ of coalgebras). The assumption implies that a morphism is always continuous.
\end{defi}

\begin{ex}
Finite dimensional Hopf algebras (coalgebras) are precisely the discrete proartinian Hopf algebras (coalgebras). Remark that the projection morphisms $\pi_U:H\to H/U$, with $U\in\Lambda_H$ are morphism of proartinian Hopf algebras (coalgebras), with respect to this discrete topology on the quotients.
\end{ex}

\begin{ex}
Let $G$ be a profinite group and $\Omega_G$ be a system of open neighbourhoods of the identity element of~$G$ consisting of normal subgroups.
Then the {\em completed group ring}
\[ K \llbracket G\rrbracket = \plim{N \in \Omega_G} K[G/N] \]
is a proartinian Hopf algebra over~$K$.
\end{ex}

Proartinian coalgebras and proartinian Hopf algebras $H$ inherit additional structure from the structures making each of the $H/U$ a coalgebra or a Hopf algebra, respectively.
For instance, in the case of a proartinian Hopf algebra, we obtain a $K$-linear multiplication on~$H$, which is obtained from the universal property of $H$ as a projective limit from the multiplications $m_U$ on each $H/U$, by means of the following commutative diagram for any $U \in \Lambda_H$:
\[
\xymatrix{
H\ot H \ar[rr]^-m \ar@{->>}[d]_{\pi_U\ot \pi_U} && H \ar@{->>}[d]^-{\pi_U}\\
H/U\ot H/U \ar[rr]^-{m_U} && H/U.
}
\]
This multiplication map turns $H$ into a (usual) $K$-algebra and moreover, the commutativity of the above diagrams implies that the multiplication map is continuous.
Further, since the multiplication map $m$ is continuous, it can be extended to a morphism with the completed tensor product as domain:
\[ \widehat m: H \hot H \to H.\]
The situation of the coalgebra structure is slightly more complicated. This time, one uses the universal property of the completed tensor product $H\hot H$ as a projective limit to define a comultiplication on $H$ that makes the following diagram commute for all $U,V,W \in \Lambda_H$ such that $U \subseteq V$ and $U \subseteq W$:
\begin{equation}\label{Deltacontinuous}
\xymatrix{
H \ar[rrr]^-\Delta \ar@{->>}[d]_-{\pi_U} &&& H \hot H \ar@{->>}[d]^-{\pi_V\hot \pi_W}\\
H/U \ar[r]^-{\Delta_U} & H/U\ot H/U \ar[rr]^-{\pi_{V,U}\ot \pi_{W,U}} && H/V\ot H/W.
}\end{equation}
Remark, however, that in this case and unless $H$ is discrete, $H$ is not a coalgebra (or Hopf algebra) in the classical sense because of the necessity to work with the completed tensor product. Rather, $H$ is a topological coalgebra in the sense of \cite{Takeuchi85}. If all $H/U$ are cocommutative, $H$ will be cocommutative as well, by definition.

Similarly, a unit morphism $u:K\to H$, a counit $\epsilon: H \to k$ and an antipode $S:H\to H$ can be constructed by taking a projective limit of the structure maps of all $H/U$ for $U \in \Lambda_H$:
\begin{equation}\label{uepsilonScontinuous}
\xymatrix{
K \ar@{=}[d] \ar[rr]^-u && H \ar@{->>}[d]^{\pi_U} \\
K \ar[rr]^{u_{U}} && H/U
} \qquad
\xymatrix{
 H \ar[rr]^-\epsilon \ar@{->>}[d]^{\pi_U} && K \ar@{=}[d] \\
H/U \ar[rr]^-{\epsilon_U} && K
} \qquad
\xymatrix{
 H \ar[rr]^-S \ar@{->>}[d]^{\pi_U} && H \ar@{->>}[d]^{\pi_U} \\
H/U \ar[rr]^-{S_U} && H/U
}
\end{equation}
These structure maps then satisfy the expected coalgebra or Hopf-algebra axioms (see, for instance, \cite[p.~36]{Fontaine}), which implies that $H$ is a coalgebra or Hopf algebra in the monoidal category of proartinian vector spaces (with completed tensor product).

The projective limit $H = \plim{U \in \Lambda_H} H/U$ in $\Vect_K$ differs from (in fact, is strictly bigger than) the projective limit computed in the category of all (cocommutative) Hopf algebras. More precisly, the projective limit in the category of cocommutative Hopf algebras is exactly the biggest subspace of $H$ for which the comultiplication $\Delta:H\to H\hot H$ factors through the usual tensor product $H\ot H$.

Since we suppose all the Hopf algebras $H_U := H/U$ for $U \in \Lambda_H$ to be finite dimensional, their linear dual is again a Hopf algebra and this gives rise to an inductive system of Hopf algebras $(H^*_U,\pi^*_{V,U})$. The inductive limit of this system in the category of Hopf algebras coincides with the one in the category of $K$-vector spaces, and hence provides a (usual) Hopf algebra $H^\dagger=\ilim{U \in \Lambda_H} H^*_U$. Using an argument based on adjunctions, one can easily verify that the Sweedler dual of $H^\dagger$ is exactly the projective limit of the system $(H_U,\pi_{V,U})$ computed in the category of Hopf algebras, and the full linear dual of $H^\dagger$ is isomorphic to $H$ (as $K$-vector space).

\begin{lem}\label{lem:isoartcoalg}
Let $H$ and $H'$ be proartinian $K$-coalgebras and $f:H\to H'$ a $K$-linear map. 
\begin{enumerate}[(a)]
\item \label{lem:isoartcoalg:a}
Then the following two conditions are equivalent:
\begin{enumerate}[(i)]
\item \label{lem:isoartcoalg:a:ii} $f$ is a morphism of proartinian coalgebras;
\item \label{lem:isoartcoalg:a:i} $f$ is continuous and $\Delta_{H'}\circ f=(f\hot f)\circ \Delta_H$ and $\epsilon_{H'}\circ f=\epsilon_H$.
\end{enumerate}
\item \label{lem:isoartcoalg:b}
If $f$ is a morphism of proartinian coalgebras and it is moreover a homeomorphism, then $f^{-1}$ is again a morphism of proartinian coalgebras.
\end{enumerate}
\end{lem}

\begin{proof}
\eqref{lem:isoartcoalg:a} Since morphisms of proartinian coalgebras are continuous, we can assume without loss of generality that $f$ is continuous.
Now take any opens $V',W'\in \Lambda_{H'}$, and any open $U'\subset V'\cap W'$. The continuity of $f$ implies that there exists an open $U\in\Lambda_H$ such that $f(U)\subset U'$, leading to induced maps $f_{U',U}:H/U\to {H'}/{U'}$.
Now consider the following diagram:
\[
\xymatrix{
H\ar[rrr]^{f} \ar[dd]^{\Delta_H} \ar[dr]^{\pi_{U}} &&& {H'}\ar[dd]^(.65){\Delta_{H'}}\ar[dr]^{\pi_{U'}} && \\
& H/U \ar[rrr]^(.35){f_{U',U}}\ar[dd]^(.35){\Delta_{H/U}}   &&& {H'}/{U'} \ar[dd]^{\Delta_{{H'}/{U'}}} & \\
H \hot H \ar[rrr]^(.65){f \hot f} \ar[dr]^{\pi_{U} \hot \pi_{U}}  &&& {H'} \hot {H'} \ar@(d,l)[dddr]_{\pi_{V'}\hot \pi_{W'}} \ar[dr]^{\pi_{U'} \hot \pi_{U'}}&& \\
& H/U \ot H/U \ar[rrr]^(.4){f_{U',U} \ot f_{U',U}} &&& {H'}/{U'} \ot {H'}/{U'} \ar[dd]^{\pi_{V',U'}\ot \pi_{W',U'}} & \\
\\
&&
&& {H'}/{V'} \ot {H'}/{W'} \\
}
\]
The continuity of $f$ and, as a consequence, of $f \hot f$ implies that the top and the bottom of the box are commutative diagrams.
The left and the (extended) right side walls commute by the construction of $\Delta_H$ and $\Delta_{H'}$, using that $H$ and ${H'}$ are proartinian coalgebras, see \eqref{Deltacontinuous}.
The consequence we are interested in is the following. Since each of the maps $\pi_U$ is surjective, and the morphisms $\pi_{V'}\hot \pi_{W'}$ (varying over all $V',W'\in \Lambda_{H'}$) are jointly monic, the back of the box commutes, i.e.\ $\Delta_{H'}\circ f=(f\hot f)\circ \Delta_H$, if and only if the front half of the box commutes for all $U'$ (and correspondingly chosen $U$), that is, if and only if the maps $f_{U',U}$ respect the comultiplications.

Further consider this diagram.
\[
\xymatrix{
H    \ar[rr]^{f}        \ar[d]^{\pi_{U}}  \ar@/_7pc/[dddr]_{\epsilon_H}  && {H'}   \ar[d]_{\pi_{U'}} \ar@/^7pc/[dddl]^{\epsilon_{H'}} \\
H/U \ar[rr]^{f_{U',U}}
\ar[ddr]_{\epsilon_{H/U}}&& {H'}/{U'}
\ar[ddl]^{\epsilon_{{H'}/{U'}}} \\
\\
& K & \\
}
\]
The rectangle is commutative by the continuity of~$f$. The left half is commutative by the construction of $\epsilon_H$, coming from the proartinian coalgebra structure, and similarly also the right half is commutative by the construction of~$\epsilon_{H'}$.
Using the surjectivity of each $\pi_U$, we see that $\epsilon_H = \epsilon_{H'} \circ f$ if and only if the lower part of the diagram (the diagram without the first line) commutes, that is, if and only if the maps $f_{U',U}$ respect the counits.

The proof is finished because the commutativity of the front half of the box and of the above diagram without the first line is exactly the definition of a morphism of proartinian algebras.

\eqref{lem:isoartcoalg:b} If $f$ is a morphism of proartinian coalgebras and a homeomorphism, then $f^{-1}$ clearly satisfies $\Delta_H\circ f^{-1} = (f^{-1}\hot f^{-1})\circ \Delta_{H'}$, and hence $f^{-1}$ is again a morphism of proartinian coalgebras by part~\eqref{lem:isoartcoalg:a}.
\end{proof}

As in the finite case, a crucial role will be played by grouplike elements.

\begin{defi}
An element $x$ in a proartinian coalgebra $H$ is called {\em grouplike} if and only if $\pi_U(x)$ is grouplike in the (finite dimensional) coalgebra $H/U$ for any $U\in \Lambda_H$. 
\end{defi}

\begin{ex}\label{ex:grouplikes}
Consider a profinite set $X= \plim{i\in I} X_i$ taken as the limit of a projective system formed by maps of finite sets $\pi_{ij}: X_j\to X_i$ for any $i,j\in I$ with $i\le j$. Then define $H=K\llbracket X\rrbracket=\plim{i\in I} K[X_i]$, which is a proartinian coalgebra by considering a coalgebra structure on each finite dimensional $K[X_i]$, where $X_i$ is a base of grouplike elements. Consider now any grouplike element $x\in H$. This means that for any $i\in I$, $\pi_i(x)\in K[X_i]$ is grouplike, that is, $\pi_i(x)\in X_i$. Therefore, by definition of $X$ as profinite set, we find that $x\in X$. We can conclude that the grouplike elements of $K\llbracket X\rrbracket$ are exactly the elements of $X$. Remark that these elements are a linearly independent set, but not a base for $K\llbracket X\rrbracket$ unless $X$ is finite.
\end{ex}

\begin{cor}\label{co:grouplikes}
The grouplike elements of a proartinian coalgebra $H$ form a profinte set $G(H)$.

Any morphism of proartinian coalgebras restricts to a continuous map between the respective profinite sets of grouplike elements.

In fact, taking grouplikes leads to a functor from the category of proartinian coalgebras to the category of profinite sets, which is a right adjoint to the fully faithful functor that sends a profinite set $X= \plim{i\in I} X_i$ to the proartinian coalgebra $K\llbracket X\rrbracket=\plim{i\in I} K[X_i]$.
\end{cor}

\begin{proof}
Let $H$ be a proartinian coalgebra, and take any $U\in \Lambda_H$. Then $H/U$ is a finite dimensional coalgebra. Since grouplike elements are linearly independent, the set $G(H/U)$ of grouplike elements in $H/U$ is finite. Moreover, since coalgebra morphisms preserve grouplike elements, we find that the projection morphisms $\pi_{V,U}:H/U\to H/V$ induce maps $\gamma_{V,U}:G(H/U)\to G(H/V)$ by restriction. Therefore we obtain a projective system of finite sets, whose projective limit is exactly the set of grouplike elements in $H$. 

The second statement follows directly from the first one.

For the finial statement recall from Example~\ref{ex:grouplikes} that $G(K\llbracket X\rrbracket)=X$, which is the unit of the desired adjunction. By the second statement, we then find that any coalgebra morphism from $K\llbracket X\rrbracket$ to a proartinian coalgebra $H$ is completely determined by a morphism from $X$ to $G(H)$, showing that we indeed have an adjunction as stated, and since the unit is an isomorphism, the left adjoint is fully faithful.
\end{proof}

Next we turn our attention to the theory of modules and (Hopf) ideals in proartinian Hopf algebras as well as coideals in proartinian coalgebras.
First note that the usual notion of left/right/two-sided ideal makes sense for proartinian Hopf algebras.
However, since the comultiplication in a proartinian coalgebra is defined by a completed tensor product, we have to adapt the notion of coideal accordingly, which we do as follows.

\begin{defi}
Let $H$ be a proartinian coalgebra over~$K$ with basis of neighbourhoods~$\Lambda_H$ as in Definition~\ref{defi:infHopf}.
We say that a $K$-subspace $I \subseteq H$ is a {\em two-sided proartinian coideal} if $\pi_U(I)=(I+U)/U \subseteq H/U$ is a two-sided coideal for all $U \in \Lambda_H$.

Similarly, if $H$ is a proartinian Hopf algebra over~$K$ with $\Lambda_H$ as in Definition~\ref{defi:infHopf}, then a $K$-subspace $I \subseteq H$ is a {\em proartinian Hopf ideal} if $(I+U)/U \subseteq H/U$ is a Hopf ideal for all $U \in \Lambda_H$.
\end{defi}

The next lemma shows that ideals in the usual sense coincide with closed ``proartinian'' ideals.

\begin{lemma}\label{le:proartinianideals}
Let $H$ be a proartinian Hopf algebra.
\begin{enumerate}[(a)]
\item\label{le:proartinianideals:a} If $I$ is a left ideal in $H$ (in the classical sense), then for each $U\in \Lambda_H$, we have that $(I+U)/U$ is a left $H/U$-ideal. Hence, the action of $H$ on $I$ is continuous with respect to the subtopology on $I$.
\item\label{le:proartinianideals:b} If $I\subset H$ is such that for each $U\in \Lambda_H$, $(I+U)/U$ is a left $H/U$-ideal, then the closure $\widehat I$ is a left ideal in $H$.
\item\label{le:proartinianideals:c} If $I\subset H$ is closed, then $I$ is a (closed) left ideal in $H$ if and only if $(I+U)/U$ is a left ideal in $H/U$ for all $U\in \Lambda_H$. 
\end{enumerate}
\end{lemma}

\begin{proof}
\eqref{le:proartinianideals:a} Since by definition, the multiplication of $H$ induces a multiplication on all finite quotients by open Hopf ideals, the statement follows by the commutativity of the following diagram since $\pi_U(I)=(I+U)/U$ for all $U\in \Lambda_H$.
\[
\xymatrix{
H\times I \ar[rr]^-m \ar@{^(->}[d] && I \ar@{^(->}[d]\\
H\times H \ar[rr]^-m \ar@{->>}[d]^-{\pi_U\times \pi_U} && H \ar@{->>}[d]^-{\pi_U}\\
H/U \times H/U \ar[rr]^-{m_U} && H/U
}
\]
\eqref{le:proartinianideals:b} The statement tells us that the multiplication maps of each finite quotient $H/U$ restrict to a left action of $H/U$ on $(I+U)/U$:
\[\xymatrix{
H/U \times (I+U)/U \ar[rr]^-{m_U} && (I+U)/U
}
\]
Taking the projective limit along all $U\in \Lambda_H$, we obtain the desired action of $H=\plim{U\in \Lambda_H} H/U$ on $\widehat I=\plim{U\in\Lambda_H} (I+U)/U$.\\
\eqref{le:proartinianideals:c} is a direct combination of the previous two points since $I=\widehat I$ in this case.
\end{proof}

In contrast to the above theorem, a proartinian coideal is not a coideal in the classical sense. The reason is similar to the observation made above that the comultiplication of a proartinian coalgebra lands in the completed tensor product, and hence a proartinian coalgebra is not a coalgebra in the usual sense. By the same arguments, one can observe that
a proartinian coideal $I$ in $H$ has the property that $\Delta(I)\subset I\hot H +  H\hot I$.

\begin{defi}
Let $H$ be a proartinian Hopf $K$-algebra.
A topological $K$-vector space $L$ is called a {\em left $H$-module} if there is a continuous $K$-linear map
\[ H \times L \to L, \;\;\; (h,x) \mapsto h.x \]
such that
\begin{enumerate}[(1)]
\item $1.x = x$ for all $x \in L$ and
\item $(h_1 h_2). x = h_1.(h_2.x)$ for all $x \in L$ and all $h_1,h_2 \in H$.
\end{enumerate}
Further, a {\em proartinian $H$-module coalgebra over~$K$} is a proartinian left $H$-module which carries a compatible structure of proartinian coalgebra.
Finally, a {\em proartinian $H$-module algebra over~$K$} is a proartinian left $H$-module which carries a compatible structure of proartinian algebra.
\end{defi}

The following lemma, which will be useful later, now follows directly by combining the preceding definitions with the properties of the proartinian topology as stated in Proposition~\ref{prop:ProVect}.

\begin{lemma}\label{lem:openclosedideal}
Let $H$ be a proartinian coalgebra over $K$, and $I\subset H$ a $K$-linear subspace. 
\begin{enumerate}[(a)]
\item \label{lem:openclosedideal:a} $I$ is a closed (and open) proartinian coideal if and only if $H/I$ is a proartinian (and finite dimensional) coalgebra and the natural projection $H \twoheadrightarrow H/I$ is continuous.
\end{enumerate}
Suppose now that moreover $H$ is a proartinian Hopf algebra over $K$.
\begin{enumerate}[(a)]
\setcounter{enumi}{1}
\item \label{lem:openclosedideal:b} $I$ is a closed (and open) left ideal proartinian coideal in $H$ if and only if $H/I$ is a proartinian (and finite dimensional) left $H$-module coalgebra and the natural projection $H \twoheadrightarrow H/I$ is continuous.
\item \label{lem:openclosedideal:c} $I$ is a closed (and open) proartinian Hopf ideal in $H$ if and only if $H/I$ is a proartinian (and finite dimensional) Hopf algebra and the natural projection $H \twoheadrightarrow H/I$ is continuous.
\end{enumerate}
In particular, one can choose the set $\Lambda_H$ from Definition~\ref{defi:infHopf} to be the set of all open proartinian coideals in a proartinian coalgebra and the set of all open proartinian Hopf ideals in a proartinian Hopf algebra.
\end{lemma}

To finish this section, we investigate more closely the case of {\em discrete} left $H$-modules.

\begin{lem}\label{lem:LI}
Let $H$ be a proartinian Hopf $K$-algebra and consider a $K$-linear map
$$\cdot: H\times L\to L.$$ 
Then the above map is continuous with respect to the discrete topology on $L$ (in particular, this is the case if $L$ is a discrete $H$-module) if and only if
\[ L = \bigcup_{I \subset H \textrm{ open proartinian Hopf ideal}} L^I, \]
where $L^I = \{ x \in L \;|\; \forall\, i \in I : i \cdot x = 0\}$.
\end{lem}

\begin{proof}
Since $L$ carries the discrete topology, the continuity of the action is equivalent to
\[ \{(h,y) \in H \times L \;|\; h \cdot y = x \}\]
being an open set in $H \times L$ for all $x \in L$. This set is the disjoint union of all the sets
\[ \{h \in H  \;|\; h \cdot y = x \} \times \{y\}  \]
for $y \in L$. Consequently, continuity is equivalent to $\{h \in H  \;|\; h \cdot y = x \}$ being open in~$H$ for all $x,y \in L$.
To characterise this, let $\bar h \in H$ be any element such that $\bar h \cdot y=x$. Then
\[ \{h \in H  \;|\; h \cdot y = x = \bar h \cdot y \} = \bar h + \{i \in H \;|\; i \cdot y=0 \}. \]
Since addition on $H$ is by homeomorphisms, continuity is equivalent to the left ideal
\[ I_y = \{i \in H \;|\; i \cdot y=0 \} \]
being open in~$H$.
By the proartinian topology on $H$, this is further equivalent to the assertion that, for every $y \in L$, $I_y$ contains an open proartinian Hopf ideal $I \subset H$ (and hence $y\in L^{I_y}\subset L^I$). This is then further equivalent to the assertion that every $y \in L$ lies in $L^I$ for some open proartinian Hopf ideal $I \subset H$, as claimed.
\end{proof}

\subsection{Hopf-Galois extensions}

Let $L/K$ be a separable field extension and $H$ a proartinian Hopf $K$-algebra such that $L$ is a discrete left $H$-module algebra.
Then we can see $\End_K(L)$ naturally as a projective limit of $K$-vector spaces in the following way. Using Lemma~\ref{lem:LI} in the first equality, we find that
\[\End_K(L)=\Hom_K\left(\bigcup_I L^I,L\right)=\plim{I} \Hom_K(L^I,L)\]
where $I$ varies over all open proartinian Hopf ideals in $H$.
If each $\Hom_K(L^I,L)$ is finite dimensional as $L$-vector space, equivalently, if $L^I$ is finite dimensional as $K$-vector space, then $\End_K(L)$ will be proartinian, but in general this is not guaranteed at this point.

We can define, by means of the universal property of $L\hot_K H$ as projective limit, the proartinian version of the canonical Hopf-Galois map as the morphism $\can$ which makes the following diagram commutative for any open proartinian Hopf ideal $I\subset H$
\begin{equation}\label{defcanproartinian}
\xymatrix{
L \hot_K H \ar[rr]^-\can \ar@{->>}[d]^{\pi_I} && \End_K(L) \ar@{->>}[d]^{\res_I}\\
L \ot_K H/I \ar[rr]^-{\can_I} && \Hom_K(L^I,L)
}\end{equation}
where, as in the finite case, we define $\can_I(x\ot \ol h)(y)=x(\ol h\cdot y)$ for all $x \in L$, $y\in L^I$ and $\ol h\in H/I$.
By construction, this morphism is then continuous.

\begin{lem}\label{propscan}
Let $L/K$ be an $H$-Galois extension as in Definition \ref{definfiniteHG}. Consider any open proartinian Hopf ideal $I\subset H$. Then
\begin{enumerate}[(a)]
\item \label{propscan:a} $\can_I$ is surjective.
\item \label{propscan:b} $L^I$ is finite dimensional as a $K$-vector space, and hence $\End_K(L)$ is a proartinian $L$-vector space.
\item \label{propscan:c} $\Hom_K(L^I,L)$ is an $L$-coalgebra whose coalgebra structure is defined by the formula
$$f(xy)=f_{(1)}(x)f_{(2)}(y)$$
for all $f\in \Hom_K(L^I,L)$ and $x,y \in L^I$. Consequently, $\End_K(L)$ is a proartinian $L$-coalgebra.
\item \label{propscan:d} $\can_I$ is a morphism of $L$-coalgebras and therefore $\can$ is an isomorphism of proartinian $L$-coalgebras.
\item \label{propscan:e} For any field extension $\tilde{L}/L$, 
$\Hom_K(L,\tilde L)$ is a right proartinian $\tilde L\hot_K H$-module $\tilde L$-coalgebra and the canonical map $\widetilde\can: \tilde{L} \hot_K H \to \Hom_K(L,\tilde{L})$ is a morphism of right $\tilde L\hot_K H$-module $\tilde L$-coalgebras.
\end{enumerate}
\end{lem}

\begin{proof}
\eqref{propscan:a} follows directly from the commutative diagram \eqref{defcanproartinian} that defines $\can$.\\
\eqref{propscan:b} Since $\can_I$ is surjective and $H/I$ is finite dimensional as $K$-vector space, we find that  $\Hom_K(L^I,L)$ is finite dimensional as $L$-vector space. As remarked before the lemma, this was the only missing condition for $\End_K(L)$ to be a proartinian $L$-vector space. Furthermore as $\Hom_K(L^I,L)$ is finite dimensional as $L$-vector space, $L^I$ is also finite dimensional as $K$-vector space.\\
\eqref{propscan:c} Since $L^I$ is finite dimensional by part \eqref{propscan:b}, the first statement is exactly Lemma \ref{grouplikes}\eqref{grouplikes:c}. The second statement then follows by definition since $\End_K(L)$ is a projective limit of finite dimensional $L$-coalgebras. \\
\eqref{propscan:d} Since $H/I$ and $L^I$ are finite dimensional, $\can_I$ being a coalgebra morphism is exactly Lemma~\ref{cancoalgmap}. Therefore, the homeomorphism $\can$ is a(n iso)morphism of proartinian coalgebras by Lemma \ref{lem:isoartcoalg}(\ref{lem:isoartcoalg:b}).\\
\eqref{propscan:e} 
For any open proartinian Hopf ideal $I$ in $H$, we know from \eqref{propscan:b} that $L^I$ is finite dimensional over $K$ and hence $\Hom_K(L^I,\widetilde L)$ is finite dimensional over $\widetilde L$. Therefore, we can apply Lemma~\ref{cancoalgmap} and see that $\Hom_K(L^I,\widetilde L)$ is a $\tilde L\hot_K H/I$-module $\tilde L$-coalgebra  and
 $\widetilde\can_I:\widetilde L \hot_K H/I\to \Hom_K(L^I,\widetilde L)$ is a morphism of right $\tilde L\hot_K H/I$-module $\tilde L$-coalgebras. 
We can then conclude by taking the projective limit over the set~$\Lambda_H$ of all open proartinian Hopf ideals of $H$.
\end{proof}

Remark that even though $\can$ is an isomorphism of proartinian coalgebras, it does not follow immediately that $\can_I$ is a bijection for any open proartinian Hopf ideal $I$ (see Example~\ref{ex:non-iso}).
The difficulty is to establish the injectivity of $\can_I$, so that we can link a proartinian Hopf-Galois extension with finite dimensional ones. To that aim, we will extend Greither--Pareigis theory to the infinite case in the next subsection.
We first establish some properties of $J(L_0)$, defined in~\eqref{defi:JL0}.

\begin{lemma}\label{le:propertiesJ(L0)}
For any intermediate field $K\subset L_0\subset L$, we find that
\begin{enumerate}[(a)]
\item\label{le:propertiesJ(L0):a} $J(L_0)$ is a closed left ideal two-sided proartinian coideal of $H$.
\item\label{le:propertiesJ(L0):b} If $L_0$ is $H$-stable, then $J(L_0)$ is also a right ideal. 
\item\label{le:propertiesJ(L0):c} The canonical map $\can_0:L \hot H/J(L_0) \to \Hom_K(L_0,L)$ is a quotient map.
\item\label{le:propertiesJ(L0):d} $L_0$ is an $H$-subextension if and only if $\can_0$ is injective.
\end{enumerate}
\end{lemma}

\begin{proof}
\eqref{le:propertiesJ(L0):a}
$J(L_0)$ is a closed set since it equals the intersection $\bigcap_{x \in L_0} m_x^{-1}\{(0)\}$ where $m_x: H \to L$ is the continuous map $h \mapsto h \cdot x$ for $x \in L$.
Since $L=\bigcup_{I\in \Lambda_H} L^I$ (see Lemma~\ref{lem:LI}), we also have $J(L_0) = \bigcap_{I \in \Lambda_H} J(L_0^I)$, and for every $I \in \Lambda_H$, we have $I \subset J(L_0^I)$ and the image of $J(L_0^I)$ in $H/I$ coincides with the left ideal two-sided coideal of $H/I$ consisting of the elements annihilating $L_0^I$, we obtain isomorphisms
\[ H/J(L_0) \cong \plim{I \in \Lambda_H} H/J(L_0^I) \cong \plim{I \in \Lambda_H} (H/I)/(J(L_0^I)/I). \]
Therefore, we can conclude by Lemma~\ref{lem:openclosedideal}\eqref{lem:openclosedideal:b} that $J(L_0)$ is a closed left ideal two-sided proartinian coideal of~$H$.

\eqref{le:propertiesJ(L0):b}
This follows in the same way as in the finite case. For any $h\in H$, $i\in J(L_0)$ and $x\in L_0$, we find that $(ih)\cdot x = i\cdot (h\cdot x) = 0$, since $h\cdot x\in L_0$, and therefore also $ih\in J(L_0)$.

\eqref{le:propertiesJ(L0):c} 
Consider the following diagram
\[\xymatrix{L \hot H \ar[r]^-{\can}_-\sim\ar@{->>}[d]& \End_K(L)\ar@{->>}[d]\\ L \hot H/J(L_0)\ar[r]^{\can_0} & \Hom_K(L_0, L).}\]
The vertical arrows are quotient maps (i.e.\ they are surjective, and topology on their image is the quotient topology). Since $\can$ is a homeomorphism, the composition is also a quotient map, and therefore $\can_0$ is a quotient map as well.

\eqref{le:propertiesJ(L0):d} 
is now obvious since any injective quotient map is a homeomorphism.
\end{proof}

We remark that for every intermediate field $K\subset L_0\subset L$ we trivially have $L_0 \subseteq L^{J(L_0)}$ and, consequently, $J(L_0) \supseteq J(L^{J(L_0)})$.
In fact, we have equality
\[ J(L_0) = J(L^{J(L_0)})\]
as the other inclusion is clear. This of course is, in other words, a direct consequence of the fact that we have a Galois connection between intermediate field extensions and (proartinian) left ideals two-sided coideals.

\begin{prop}\label{pr:subextensioninvariants}
Let $L/K$ be a Hopf-Galois extension for a proartinian Hopf algebra~$H$.
\begin{enumerate}[(a)]
\item\label{pr:subextensioninvariants:a} For any $H$-subextension, we have $L_0 = L^{J(L_0)}$.
\item\label{pr:subextensioninvariants:b} For any closed ideal two-sided proartinian coideal $C$ in~$H$ such that the canonical map
\[L \hot_K H/C \xrightarrow{\can_C} \Hom_K(L^C,L)\]
is an isomorphism, we have that $L^C$ is an $H$-subextension of~$L$ and $C = J(L^C)$.
\end{enumerate}
\end{prop}

\begin{proof}
\eqref{pr:subextensioninvariants:a}
As $L_0 \subseteq L^{J(L_0)}$, we can consider the composition
\[ \can_0: \xymatrix{ L \hot H/J(L_0) \ar@{->>}[rr]^-{\can_{J_0}} && \Hom_K(L^{J(L_0)},L) \ar@{->>}[rr]^-{\res} && \Hom_K(L_0,L)}. \]
The assumption that $L_0$ is an $H$-subextension implies that this composition is an isomorphism, whence both maps are isomorphisms.
Consequently, $L_0 = L^{J(L_0)}$.

\eqref{pr:subextensioninvariants:b}
Consider the following diagram
\[
\xymatrix{
\tilde L\hot H/C \ar[rr]^-{\can_C} \ar@{->>}[dr]_-\pi && \Hom_K(L^C,L) \\
& \tilde L\hot H/J(L^C) \ar[ur]_-{\can_{L^C}}
}
\]
Since $\can_C$ is bijective by assumption and $\pi$ is surjective, $\can_{L^C}$ is bijective, so in particular injective, hence $L^C/K$ is an $H$-subextension of~$L$ by Lemma~\ref{le:propertiesJ(L0)}.
Now $\pi$ is the composite isomorphism $\can_{L^C}^{-1} \circ \can_C$, from which we conclude $C=J(L^C)$.
\end{proof}

To get the full correspondence one then also needs that any closed ideal two-sided proartinian coideal in $H$ arises this way.
We will need results for it that will follow from infinite Greither--Pareigis theory, allowing us to show that $L^C$ is an $H$-subextension if $C$ is any closed ideal two-sided proartinian coideal in $H$.

\subsection{Infinite Greither--Pareigis theory}

The aim of this section is to connect finite and infinite Hopf-Galois theory.
Let $L/K$ be a (possibly infinite) separable extension, and $\tilde L$ be a normal closure over~$K$. Then we know by classical (infinite) Galois theory that $G=\Gal(\tilde L/K)$ is a profinite group containing $G'=\Gal(\tilde L/L)$ as a closed subgroup, and we have isomorphisms
\begin{multline*}
\Hom_{K-\Alg}(L,\tilde L) \cong  \Hom_{K-\Alg}(\bigcup_{N \lhd G \textnormal{ open}} L^{N}, \tilde{L}) \cong  \Hom_{K-\Alg}(\bigcup_{N \lhd G \textnormal{ open}} \tilde{L}^{G'N}, \tilde{L})\\
\cong \plim{N \lhd G \textnormal{ open}} G/(G'N) =  G/G'
\end{multline*}
as profinite left $G$-sets and 
\[\Hom_K(L,\tilde L)\cong \plim{N \lhd G \textnormal{ open}} \tilde L[G/(G'N)] \cong \tilde L\llbracket G/G'\rrbracket\]
as proartinian $\tilde L$-vector spaces and left $\tilde L\llbracket G\rrbracket$-modules, where $\tilde L\llbracket G\rrbracket$ is the completed group algebra.

\begin{prop}\label{prop:GPinf}
Let $L/K$ be a separable $H$-Galois extension as in Definition~\ref{definfiniteHG} and use the notation just introduced. Then

\begin{enumerate}[(a)]
\item \label{prop:GPinf:a} For any open proartinian Hopf ideal $I\subset H$, $\Hom_{K-\Alg}(L^I,\tilde L)$ is a base of grouplike elements for $\Hom_K(L^I,\tilde L)$ as $\tilde L$-coalgebra.

\item \label{prop:GPinf:b} For any open proartinian Hopf ideal $I\subset H$, $\tilde L\ot_K H/I$ is isomorphic as a Hopf $\tilde L$-algebra to a group algebra $\tilde L[N_I]$, for some finite group $N_I$.

\item \label{prop:GPinf:c} The groups $N_I$ form a projective system of finite groups when $I$ runs through the open proartinian Hopf ideals of~$H$ and hence the projective limit $N=\plim{I}N_I$ is a profinite group.

\item \label{prop:GPinf:d} $\tilde L\hot_K H$ is isomorphic as proartinian Hopf $\tilde L$-algebra to the completed group algebra $\tilde L \llbracket N\rrbracket$. Moreover the natural left action of $G$ on the first tensorant of $\tilde L\hot_K H$ induces a continuous left $G$-action on $N$.

\item \label{prop:GPinf:e} From the canonical map, we obtain a homeomorphism of profinite sets
\[\beta:N\cong G/G'.\]

\item \label{prop:GPinf:f} Since $\Hom_K(L,\tilde L)\cong \tilde L\llbracket G/G'\rrbracket$ is a right proartinian $\tilde L\hot_K H\cong\tilde L \llbracket N \rrbracket$-module $\tilde L$-coalgebra by Lemma~\ref{propscan}, we obtain a continuous right action of $N$ on $G/G'$, which satisfies the compatibility relation
\[ g.\big( (hG').n \big) = (ghG').(g.n) \]
where $n \in N$ and $g,h \in G$.
Moreover, the map $\beta$ is right $N$-equivariant and sends $n \in N$ to $(1G').n$

\item \label{prop:GPinf:g} For any open proartinian Hopf ideal $I\subset H$, there exists an open normal subgroup $V\subset N$ such that $N/V\cong N_I$ and $\tilde L\ot_K H/I\cong \tilde L[N/V]$ as Hopf $\tilde L$-algebras.
\end{enumerate}
\end{prop}

\begin{proof}
\eqref{prop:GPinf:a} By Lemma~\ref{propscan}\eqref{propscan:b} we know that $L^I$ is of finite dimension, and hence $\Hom_K(L^I, \tilde L)$ has a base of grouplike elements by separability and Lemma~\ref{grouplikes}, which are given exactly by the $K$-algebra morphisms from $L^I$ to $\tilde L$.

\eqref{prop:GPinf:b} Using a base-extension from $L$ to $\tilde L$, we obtain a canonical map
\[\widetilde \can:\tilde L\hot_K H\to \Hom_K(L,\tilde L),\]
which is a morphism of proartinian $\tilde L$-coalgebras and a homeomorphism.
Its inverse $\widetilde\can^{-1}$ is also a morphism of proartinian $\tilde L$-coalgebras by Lemma \ref{lem:isoartcoalg}. 
In particular, $\widetilde\can^{-1}$ is continuous and preserves the comultiplication. The continuity of $\widetilde\can$ implies that there exists an open proartinian Hopf ideal $I'$ in $H$ and a (surjective, coalgebra) morphism $(\widetilde\can^{-1})_I$ such that the following diagram commutes
\[
\xymatrix{
\tilde L\hot_K H \ar[d]^{\pi_{I}} && \Hom_K(L,\tilde L) \ar[ll]_-{\widetilde \can^{-1}} \ar[d]\\
\tilde L\ot H/I && \Hom_K(L^{I'},\tilde L). \ar[ll]^-{(\widetilde\can^{-1})_{I}}
}
\]
Since $(\widetilde\can^{-1})_I$ is a surjective coalgebra morphism, and $\Hom_K(L^{I'},\tilde L)$ has a base of grouplike elements by part~\eqref{prop:GPinf:a}, it follows by Proposition \ref{grouplikesHopf} that $\tilde L\ot H/I$ also has a base of grouplike elements. Hence, $\tilde L\ot H/I$, being a finite dimensional Hopf algebra with a base of grouplike elements, is a group algebra $\tilde L[N_I]$ for some finite group $N_I$.

\eqref{prop:GPinf:c} We know that that there is a projective system of Hopf algebras $\tilde L\ot H/I$. Since each of these Hopf algebras is a finite group algebra by part~\eqref{prop:GPinf:b}, and morphisms of Hopf algebras preserve grouplike elements, we obtain henceforth a projective system of finite groups.

\eqref{prop:GPinf:d} The first part follows from the fact that $L\hot_K H$ and $\tilde L\llbracket N\rrbracket$ arise as the projective limit of isomorphic projective systems of finite dimensional Hopf $K$-algebras. The second part is clear, since $G$ acts on $\tilde L\hot_K H$ by proartinian coalgebra morphisms, hence the $G$-action on $\tilde L\hot_K H$ turns grouplikes into grouplikes by Corollary~\ref{co:grouplikes} and leads to a continuous action of $G$ on $N$.

\eqref{prop:GPinf:e} Combining the above, we obtain a sequence of isomorphisms
$$\tilde L\llbracket N\rrbracket \cong \tilde L\hot H \stackrel{\widetilde\can}{\cong} \Hom_K(L,\tilde L)\cong \tilde L\llbracket G/G'\rrbracket.$$
The first isomorphism is an isomorphism of proartinian Hopf algebras. The last isomorphism is an isomorphism of proartinian coalgebras. Finally, the canonical map $\widetilde \can$ is a morphism of proartinian coalgebras as well, by Lemma \ref{propscan}. The combined isomorphism is therefore a morphism of proartinian coalgebras. Since $N$ and $G/G'$ are exactly the grouplike elements of the first and last proartinian coalgebra (see Example~\ref{ex:grouplikes}), we obtain the desired homeomorphism of profinite sets.

\eqref{prop:GPinf:f} 
As in the finite case, since $\Hom_K(L,\tilde L)$ is a right proartinian $\tilde L\hot_K H$-module $\tilde L$-coalgebra, the grouplike elements of $N$, being grouplike elements of $\tilde L\hot_K H$ act on $\Hom_K(L,\tilde L)$ by proartinian coalgebra morphisms, and therefore send grouplikes to grouplikes (see Corollary~\ref{co:grouplikes}). A $K$-linear map $f\in\Hom_K(L,\tilde L)$ is grouplike if and only if the restriction of $f$ to $L^I$ is grouplike in $\Hom_K(L^I,\tilde L)$ for any open Hopf ideal $I\in\Lambda_H$, which means that this restriction is an algebra map (see Lemma~\ref{grouplikes}\eqref{grouplikes:c}). Since $L=\bigcup_{I\in \Lambda_H} L^I$, this means that $f$ itself is an algebra map.
Since $\Hom_{K-Alg}(L,\tilde L) \cong G/G'$ by classical Galois theory, we obtain the required continuous action of $N$ on $G/G'$.
The formulae are proven exactly as in the finite case (see Theorem~\ref{th:GP}).

\eqref{prop:GPinf:g}
For any open proartinian Hopf ideal $I\in \Lambda_H$, we know by part \eqref{prop:GPinf:b} that $\tilde L\ot_K H/I\cong \tilde L[N_I]$ is a finite dimensional group algebra, and by part \eqref{prop:GPinf:d} that $\tilde L\hot_K H\cong \tilde L\llbracket N\rrbracket$.
Via these isomorphisms, the natural projection $\tilde L\hot_K H \twoheadrightarrow \tilde L\ot_K H/I$ translates to the natural projection $\tilde L\llbracket N\rrbracket \twoheadrightarrow \tilde L[N_I]$.
The latter map is induced from a surjection of groups $N \twoheadrightarrow N_I$ (cf.\ Example~\ref{ex:grouplikes} and Corollary~\ref{co:grouplikes}). Taking $V$ to be its kernel, we find $N_I\cong N/V$ and, again by Corollary~\ref{co:grouplikes}, $\tilde L\ot I\cong \tilde L \llbracket V \rrbracket$ as needed.
\end{proof}

\begin{prop}\label{prop:LIsub}
Let $L/K$ be a separable Hopf-Galois extension for a proartinian Hopf algebra~$H$. Then for any open left ideal two-sided proartinian coideal $I \subseteq H$, the canonical map
\[\can_I: L \otimes_K H/I \twoheadrightarrow \Hom_K(L^I,L). \]
is an isomorphism.
\end{prop}

\begin{proof}
We first assume that $I$ is an open proartinian Hopf ideal.
Proposition~\ref{prop:GPinf} ensures that the hypotheses of Lemma~\ref{lem:gp-act} are fulfilled.
From that lemma and Proposition~\ref{prop:GPinf}\eqref{prop:GPinf:g}, it follows that we have the $K$-linear isomorphism
\[ \tilde L \otimes_K H/I \cong \tilde L [N/V] \cong  \tilde L[G/U] \cong \Hom_K(L^U,\tilde{L}). \]
This isomorphism allows us to apply verbatim the arguments of the proof of Lemma~\ref{lem:main} in the finite case in order to conclude $L^I = L^U$ and $\can_I$ is bijective.
For the general case, when $I$ is an open left ideal two-sided proartinian coideal of~$H$, there is a proartinian Hopf ideal $J$ contained in~$I$ since the open proartinian Hopf ideals form a base of open neighbourhoods of~$0$. As $L^J$ is $H/J$-Galois by the preceding case, $\can_I$ is an isomorphism because $H/I \cong (H/J)/(I/J)$ and $L^I = (L^J)^{(I/J)}$ and we can apply the Galois correspondence of  Lemma~\ref{lem:main} for the finite Hopf algebra $H/J$ and its left ideal two-sided coideal $I/J$.
\end{proof}

\begin{cor}\label{cor:LIsub}
Let $L/K$ be a separable Hopf-Galois extension for a proartinian Hopf algebra~$H$. Then for any open proartinian Hopf ideal $I \subseteq H$, the finite extension $L^I/K$ is $H/I$-Galois and
an $H$-subextension of~$L$.
\end{cor}

\begin{proof}
By Proposition~\ref{prop:LIsub}, $L^I/K$ is $H/I$-Galois, and we conclude by Proposition~\ref{pr:subextensioninvariants}\eqref{pr:subextensioninvariants:b}.
\end{proof}

\subsection{The infinite Hopf-Galois correspondence}

\begin{cor}\label{cor:inf-sub}
Let $L/K$ be a separable Hopf-Galois extension for a proartinian Hopf algebra~$H$ and let $C$ be a closed ideal two-sided proartinian coideal in~$H$. Then the following statements hold.
\begin{enumerate}[(a)]
\item\label{cor:inf-sub:a} $L^C$ is an $H$-subextension, $J(L^C)=C$ and $L^C = \bigcup_J L^J$, where $J$ runs through the open ideals two-sided proartinian coideals in~$H$ containing~$C$.
\item\label{cor:inf-sub:b} Moreover, $C$ is a right ideal if and only if $L^C$ is $H$-stable, i.e.\ if and only if $L^C$ is {\em $H$-normal}.
\end{enumerate}
\end{cor}

\begin{proof}
\eqref{cor:inf-sub:a}
By Proposition~\ref{prop:ProVect}\eqref{prop:ProVect:a} we know
\[C  = \bigcap_{I \in \Lambda_H} (C+I).\]
As a sum of an open and a closed ideal two-sided proartinian coideal, $C+I$ is an open ideal two-sided proartinian coideal in~$H$ containing~$C$.

Due to the inclusion $\bigcup_{I \in \Lambda_H} L^{C+I} \subseteq L^C$, we have the following diagram
\begin{equation}\label{eq:inf-sub}
\xymatrix{ 
L \hot_K H \ar[rr]_\cong^{\can} \ar@{->>}[d]_-{id\hot_K \pi_C} && \End_K(L) \ar@{->>}[d]^-{res} \\
L \hot_K H/C \ar@{->}[rr]^-{\can_C} && \Hom_K(L^C,L) \ar@{->>}[rr]^-{res} && \Hom_K(\bigcup_{I \in \Lambda_H} L^{C+I},L)}
\end{equation}
By the commutativity of the left square, $\can_C$ is surjective (note that this is exactly the same argument as in Proposition~\ref{propscan}\eqref{propscan:a}).
By Proposition~\ref{prop:LIsub}, for each $I \in \Lambda_H$, the extension $L^{C+I}$ is an $H$-subextension, whence we have the canonical isomorphism
$L \ot H/(C+I) \cong \Hom_K(L^{C+I},L)$.
Taking the projective limit of these leads to the isomorphism
\[\plim{I \in \Lambda_H} L \ot H/(C+I) \cong \plim{I \in \Lambda_H} \Hom_K(L^{C+I},L) \cong \Hom_K(\bigcup_{I \in \Lambda_H} L^{C+I},L).\]
After identification of $\plim{I \in \Lambda_H} L \ot H/(C+I)$ with $L \hot H/C $, this is exactly the horizontal composition of diagram \eqref{eq:inf-sub}, which is thus an isomorphism.
This implies the desired equality $L^C = \bigcup_J L^J$ as in the statement of the corollary, as well as the fact that $\can_C$ is bijective.
We conclude by applying Proposition~\ref{pr:subextensioninvariants}\eqref{pr:subextensioninvariants:b}.

\eqref{cor:inf-sub:b}
This can be proved by direct computation as in the finite case.
First assume that $C$ is a right ideal. Then for any $c\in C$, $h\in H$ and $x\in L^C$, we find that $c\cdot (h\cdot x)=(ch)\cdot x=0$, because $ch\in C$, and hence $h\cdot x\in L^C$.
Conversely, suppose that $L^C$ is $H$-stable, then for any $h \in H$ and any $x \in L^C$, we have $hx \in L^C$; consequently, for any $c\in C$, it follows $c\cdot (h\cdot x)=(ch)\cdot x=0$. This gives $ch \in J(L^C)=C$.
\end{proof}

We are now ready to prove the infinite Hopf-Galois correspondence. To this end, let us make the following remark. If $H$ is a proartinian Hopf algebra and $L/K$ a separable $H$-Galois extension, then by Proposition~\ref{prop:LIsub}, $L^I$ is $H/I$-Galois for any open Hopf ideal $I\in \Lambda_H$. This implies by Lemma~\ref{lem:cocommutativeHopf} that $H/I$ is cocommutative, hence we can also view $H$ as a ``cocomutative proartinian Hopf algebra''. Now consider any two-sided ideal two-sided proartinian coideal $J$ in $H$. Again by the above, we know that $(J+I)/I$ is a two-sided ideal two-sided coideal in $H/I$. Since $H/I$ is cocommutative, this implies that $(J+I)/I$ is also stable under the antipode, i.e.\ a Hopf ideal. This then means by definition that $J$ is a proartinian Hopf ideal in $H$. In other words, as in the finite case, the two-sided ideal two-sided proartinian coideals and Hopf ideals coincide in our setting.

\begin{proof}[Proof of Theorem~\ref{theo:inf-main}.]
As in the finite case, we define $\Phi(L_0)=J(L_0)$ and $\Psi(I)=L^I$.
By Corollary~\ref{cor:inf-sub}, we find that the maps $\Psi$ are well-defined, and $\Psi\circ \Phi(I)=I$. On the other hand, Lemma~\ref{le:propertiesJ(L0)} tells us that $\Phi$ is well-defined and by Proposition~\ref{pr:subextensioninvariants} we have that $\Phi\circ \Psi(L_0)=L_0$.
\end{proof}

We now generalise Lemma~\ref{H/I=N/V} to the infinite case.

\begin{prop}\label{inf:H/I=N/V}
There is a bijective correspondence between the sets of:
\begin{itemize}
\item closed proartinian left ideals two-sided coideals $C \subseteq H$, and
\item $G$-equivariant closed subgroups $V \subseteq N$
\end{itemize}
which is characterised by the property that the isomorphism from Proposition~\ref{prop:GPinf} $\tilde L \hot H \cong \tilde L \llbracket N \rrbracket$ descends to an isomorphism $\tilde L \hot H/C \cong \tilde L \llbracket N/V \rrbracket$.
Moreover, $C \subset H$ is a closed proartinian Hopf ideal if and only if $V$ is a closed normal subgroup of~$N$. In that case, the above isomorphism is an isomorphism of proartinian $\tilde{L}$-Hopf algebras.
Furthermore, these statements continue to hold with `closed' replaced by `open'.
\end{prop}

\begin{proof}
Let first $C \subseteq H$ be a closed proartinian left ideal two-sided coideal.
For any $I \in \Lambda_H$, consider the left ideal two-sided coideal $C_I = \pi_I(C) \subset H/I$. By the finite case, Lemma~\ref{H/I=N/V}, there is a $G$-equivariant subgroup $V_I \subset N_I$, where, as before, $N_I$ is such that $\tilde L \ot H/I \cong \tilde L[N_I]$. As the $N_I$ and their quotients $N_I/V_I$ form projective systems, so do the $V_I$, and letting $V = \plim{I \in \Lambda_H} V_I$ we obtain a closed $G$-equivariant subgroup of~$N$ such that $\tilde L \hot H/C \cong \tilde L \llbracket N/V \rrbracket$.

Let now $V \subseteq N$ be a closed $G$-equivariant subgroup. As $N = \plim{I \in \Lambda_H} N_I$, we can put $V_I = V \cap N_I$, identifying $N/V$ with the projective limit of the $N_I/V_I$. For each $I \in \Lambda_I$, $V_I$ is then a $G$-equivariant subgroup of $N_I$ and by Lemma~\ref{H/I=N/V}, we obtain a left ideal two-sided coideal $\overline{C}_I \subset H/I$ such that $\tilde L \ot H/\overline{C}_I \cong \tilde L [N_I/V_I]$. Writing $C_I = \pi_I^{-1}(\overline{C}_I)$, which is an open and closed proartinian left ideal two-sided coideal of~$H$, we can make the closed proartinian left ideal two-sided coideal $C = \bigcap_{I \in \Lambda_H} C_I$ of~$H$. Via the projective limit, we again find $\tilde L \hot H/C \cong \tilde L \llbracket N/V \rrbracket$.

Since these two constructions pass along the finite case, where they are mutual inverses, they give the claimed bijective correspondence.
The last statement follows as open-ness (under the assumption of closed-ness) is characterised by finite dimensionality.
\end{proof}

The proof of Corollary~\ref{co:correspGPinf} follows directly from Theorem~\ref{theo:inf-main} and this proposition.

\subsection{Example}

Here we give an example of an infinite Hopf-Galois extension, which is not classically Galois.

For any $n\in\NN$, we consider the field $L_n := \QQ(\sqrt[3^n]{2})$, which is a non-normal, separable Galois extension of $\QQ$. Applying \cite[Corollary 2.6]{GreitherPareigis87}, we find that $L_n/\QQ$ is Hopf-Galois for some Hopf algebra, and in fact it is even ``almost classically Galois'' (see \cite[Definition 4.2]{GreitherPareigis87}). Clearly, $L_n\subset L_{n+1}$ for all $n\in \NN$ and we define $L=\QQ(\sqrt[3^\infty]{2})=\bigcup_{n\in \NN} L_n$. Since moreover by \cite[Theorem 5.2]{GreitherPareigis87}, an almost classical Galois extension has a Hopf-Galois structure such that the strong structure theorem holds, we find that for a fixed $n\in \NN$, there exists a Hopf algebra $H_n$ such that $L_n$ is $H_n$-Galois and any $L_m$ with $m\le n$ is $H_m$-Galois for some Hopf algebra $H_m= H_n/I_{n,m}$, where $I_{n,m}$ is a Hopf ideal in $H_n$. Invoking the axiom of choice, we obtain for each $n\in \NN$ a Hopf algebra $H_n$ turning $L_n/\QQ$ into an $H_n$-Galois extension, such that for each $n,m\in \NN$ with $m\le n$, there is a surjective Hopf algebra morphism $\pi_{m,n}:H_n\to H_m$. In this way, we obtain a projective system of finite dimensional Hopf algebras, and we can define $H=\plim{n} H_n$ the associated proartinian Hopf algebra, and $L/\QQ$ is $H$-Galois in our sense.

\subsection{Formulation in terms of coactions}

We already remarked that if $H$ is a profinite Hopf algebra, then $H^\dagger = \ilim{U \in \Lambda_H} (H/U)^*$ is a usual Hopf algebra. The aim of this section is to show that an infinite Hopf-Galois extension in our sense is exactly a Hopf-Galois extension of $H^\dagger$ with respect to a co-action (instead of a continuous action).

\begin{proposition}\label{prop:coaction}
Let $H$ be a proartinian Hopf algebra, and $H^\dagger = \ilim{U \in \Lambda_H} (H/U)^*$. For any $K$-vector space (respectively $K$-algebra) $L$, there is a bijective correspondence between
\begin{enumerate}[(a)]
\item\label{prop:coaction:a} continuous $K$-linear maps $H\ot L\to L$ turning $L$ into a discrete $H$-module (algebra);
\item\label{prop:coaction:b} $K$-linear maps $\rho:L\to L\ot H^\dagger, \rho(x)=x_{[0]}\ot x_{[1]}$ turning $L$ into an $H^\dagger$-comodule (algebra).
\end{enumerate}
Moreover, under this correspondence, $L$ is $H$-Galois (in the sense of Definition~\ref{definfiniteHG}) if and only if $L$ is $H^\dagger$-Galois in the sense of Kreimer-Takeuchi, that is, the canonical map
$$\can^\dagger: L\ot_K L\to L\ot_K H^\dagger,\ \can^\dagger(x\ot y)=xy_{[0]}\ot y_{[1]}.$$
is a $K$-linear isomorphism.
\end{proposition}

\begin{proof}
From \eqref{prop:coaction:a} to \eqref{prop:coaction:b}.
Suppose first that $L$ has a continuous $H$-action. The we know by Lemma~\ref{lem:LI}, that $L=\bigcup_{I\in \Lambda_H} L^{I}$. In other words, for $I\subset J \in \Lambda_H$, we have a canonical inclusion $L^J\subset L^I$ and $L$ is the inductive limit of this direct system $L\cong \ilim{I\in \Lambda_H} L^I$.
By construction, the action of $H$ on $L$ induces an action $H/I\ot L^I\to L^I$ for any $I\in \Lambda_H$. Since $H/I$ is finite dimensional, we can take its dual coaction as in \eqref{coaction}, $\rho_I : L^I \to L^I\ot (H/I)^*$. Taking the inductive limit of these coactions we obtain (using the fact that the tensor product preserves inductive limits) a global coaction $\rho:L\to L \ot H^\dagger$.

From \eqref{prop:coaction:b} to \eqref{prop:coaction:a}.
Conversely, suppose that $\rho:L\to L\ot H^\dagger$ exists. Take any $x\in L$ and write $\rho(x)=\sum_{i=1}^n x_i\ot h^\dagger_i$. Since we have a finite number of elements $h^\dagger_i$ in the inductive limit $H^\dagger$, there exists some $I\in \Lambda_H$ such that $h^\dagger_i\in (H/I)^*$ for all $i=1,\ldots,n$. We then define for any $h\in H$, $h\cdot x:= \sum_i h_i^\dagger (\pi_I(h)) x_i $. One easily checks that this definition is independent of the choice of $I$, and the coassociativity and counitality of $H$ imply that this indeed defines an action of $H$ on $L$. By construction, we have that $x\in L^I$ if $\rho(x)\in L\ot (H/I)^*$. Since (as already observed above) for any $x\in L$ there exists some $I$ such that $\rho(x)\in L\ot (H/I)^*$, so $x\in L^I$, we find that $L=\bigcup_{I\in \Lambda_H} L^{I}$ and therefore the action is continuous with respect to the discrete topology on $L$ by Lemma~\ref{lem:LI}.

One checks that $\rho$ is an algebra map if and only if $L$ is an $H$-module algebra.

Now suppose that $L$ is $H$-Galois. Then by Corollary~\ref{cor:LIsub}, $\can_I: L \ot H/I \to \Hom_K(L^I,L)$ is bijective, which implies that $L^I$ is finite dimensional since $H/I$ is so. Taking the $L$-linear dual of the canonical map, we obtain the canonical map
$\can_I^*: L\ot L^I\to L\ot (H/I)^*$. By passing to the direct limit over $\Lambda_H$, we then find that $\can^\dagger:L\ot L \to L\ot H^\dagger$ is bijective.

Conversely, suppose that $\can^\dagger$ is bijective and fix any $I\in \Lambda_H$. Observe that for any $x,y\in L$, by the left $L$-linearity, $\can^\dagger(x\ot y)\in L\ot (H/I)^*$ implies that $\can^\dagger(1\ot y)\in L\ot(H/I)^*$, so  $\rho(y)\in L\ot (H/I)^*$. The latter implies that $y\in L^I$ as we already showed earlier in this proof. Therefore, the restriction of $(\can^\dagger)^{-1}$ to $L\ot (H/I)^\dagger$ lands in $L\ot L^I$, which means that it gives an inverse for $\can_I$, which is then also bijective. Since all $\can_I$ are bijective, $\can$ is bijective as well, by passing to the projective limit over $I\in \Lambda_H$.
\end{proof}

\end{document}